\documentclass[12pt]{article}
\usepackage{amssymb,amsmath,amscd, amsthm,verbatim,mathrsfs,graphicx}
\usepackage[english]{babel}
\usepackage[hidelinks]{hyperref}
\numberwithin{equation}{section}
\newtheorem{theorem}{Theorem}[section]
\newtheorem{lemma}{Lemma}[section]

\newtheorem{proposition}{Proposition}[section]

\renewcommand{\figurename}{\textbf{Fig.}}

\begin{document}
\baselineskip=14pt

\newcommand{\la}{\langle}
\newcommand{\ra}{\rangle}
\newcommand{\psp}{\vspace{0.4cm}}
\newcommand{\pse}{\vspace{0.2cm}}
\newcommand{\ptl}{\partial}
\newcommand{\dlt}{\delta}
\newcommand{\sgm}{\sigma}
\newcommand{\al}{\alpha}
\newcommand{\be}{\beta}
\newcommand{\G}{\Gamma}
\newcommand{\gm}{\gamma}
\newcommand{\vs}{\varsigma}
\newcommand{\Lmd}{\Lambda}
\newcommand{\lmd}{\lambda}
\newcommand{\td}{\tilde}
\newcommand{\vf}{\varphi}
\newcommand{\yt}{Y^{\nu}}
\newcommand{\wt}{\mbox{wt}\:}
\newcommand{\rd}{\mbox{Res}}
\newcommand{\ad}{\mbox{ad}}
\newcommand{\stl}{\stackrel}
\newcommand{\ol}{\overline}
\newcommand{\ul}{\underline}
\newcommand{\es}{\epsilon}
\newcommand{\dmd}{\diamond}
\newcommand{\clt}{\clubsuit}
\newcommand{\vt}{\vartheta}
\newcommand{\ves}{\varepsilon}
\newcommand{\dg}{\dagger}
\newcommand{\tr}{\mbox{Tr}}
\newcommand{\ga}{{\cal G}({\cal A})}
\newcommand{\hga}{\hat{\cal G}({\cal A})}
\newcommand{\Edo}{\mbox{End}\:}
\newcommand{\for}{\mbox{for}}
\newcommand{\kn}{\mbox{ker}}
\newcommand{\Dlt}{\Delta}
\newcommand{\rad}{\mbox{Rad}}
\newcommand{\rta}{\rightarrow}
\newcommand{\mbb}{\mathbb}
\newcommand{\lra}{\Longrightarrow}
\newcommand{\X}{{\cal X}}
\newcommand{\Y}{{\cal Y}}
\newcommand{\Z}{{\cal Z}}
\newcommand{\U}{{\cal U}}
\newcommand{\V}{{\cal V}}
\newcommand{\W}{{\cal W}}
\newcommand{\sta}{\theta}
\setlength{\unitlength}{3pt}
\newcommand{\msr}{\mathscr}
\newcommand{\wht}{\widehat}
\newcommand{\mfk}{\mathfrak}
\renewcommand{\figurename}{\textbf{Fig.}}

\begin{center}{\large \bf Non-canonical Oscillator Representations of Orthogonal
\\ \pse Lie Algebras and Pfaffian Varieties} \footnote {2010 Mathematical Subject
Classification. Primary 17B10, 22E47; Secondary 14M12, 35C05.}
\end{center}

\vspace{0.2cm}

\begin{center}{\large Hengjia Zhang}\end{center}

\begin{center}{
Chern Institute of Mathematics and LPMC,  \\ Nankai University, Tianjin 300071, China}\end{center}

\begin{center}
Corresponding author: \href{mailto:9820250151@nankai.edu.cn}{9820250151@nankai.edu.cn}\\
ORCID: \href{https://orcid.org/0000-0002-8652-8353}{0000-0002-8652-8353}
\end{center}

\begin {abstract}
\quad
Associated varieties are fundamental geometric invariants in the study of infinite-dimensional representations of semisimple Lie algebras. Luo and Xu (2013) constructed a large family of infinite-dimensional irreducible representations of orthogonal Lie algebras by applying Fourier transforms to the natural oscillator representations. In this paper, we prove that the associated varieties of these irreducible representations are isomorphic to explicitly described coordinate linear sections of Pfaffian varieties. The modules are realized as spaces of polynomial solutions of transformed Laplace equations. We construct filtrations from weighted polynomial degrees and determine the annihilator ideals of the corresponding associated graded modules. The defining equations consist of linear coordinate equations, minors and quadratic Pfaffian relations. The even and odd orthogonal Lie algebras are treated separately. The associated varieties are viewed as reduced affine schemes over a field of characteristic zero. In each irreducible family with fixed Fourier parameters, the associated variety is independent of the grading parameter. When the two Fourier parameters coincide, the projectivizations in the even case are Grassmannians.\psp

\noindent{\it Keywords}:\hspace{0.3cm} orthogonal Lie algebra; oscillator
 representation; associated variety; irreducible module;  Pfaffian variety; Grassmannian.

\end{abstract}

\section {Introduction}
Associated varieties are geometric invariants of finitely generated modules over universal enveloping algebras. Their relation to primitive ideals and nilpotent orbits has been studied by Borho and Brylinski, Joseph and Vogan; see \cite{BB,Ja1,Ja2,Ja3,Vd3}. Explicit computations for highest weight modules are given, for example, in \cite{BHXZ,BXX,BMW}.

By applying Fourier transformations to the canonical bosonic oscillator representations of orthogonal Lie algebras, Luo and Xu \cite{LX} constructed a large family of irreducible representations of orthogonal algebras realized on spaces of homogeneous solutions to the Laplace equation.
In this paper, we give an explicit and complete description of the associated varieties for this family of representations.
Our main results show that these associated varieties are explicitly realized as coordinate linear sections of Pfaffian varieties. When $n_1=n_2$, their projectivizations in the even case are Grassmannians.
This work continues the line of investigation initiated in \cite{ZX1}, where we established a correspondence between a certain class of irreducible $\mathfrak{sl}(n)$-modules and determinantal varieties. Here we consider orthogonal Lie algebras.
Below, we provide a more detailed technical introduction to our work.\pse

Let $\mathbb{F}$ be a field of characteristic zero, and let $n>1$ be an integer. Denote by $\msr{A} = \mathbb{F}[x_1, x_2, \dots, x_n, y_1, y_2, \dots, y_n]$ the polynomial algebra in $2n$ variables. Let $E_{r,s}$ be the square matrix with 1 as its $(r,s)$-entry and 0 elsewhere. The split even orthogonal Lie algebra is:
\begin{equation*}
{\mfk o}(2n)=\sum_{1\leq p<q\leq
n}[\mbb{F}(E_{p,n+q}-E_{q,n+p})+\mbb{F}(E_{n+p,q}-E_{n+q,p})]+\sum_{i,j=1}^{n}\mbb{F}(E_{i,j}-E_{n+j,n+i})\end{equation*}
with the commutator as its Lie bracket. Let
\begin{equation*}
H= \sum_{i =1}^{n}\mbb{F}(E_{i,i}-E_{n+i,n+i})\end{equation*} be a Cartan subalgebra and define $\{ \es_i \mid 1\leq i\leq n\} \subset H^*$ by
\begin{equation*}\es_i(E_{j,j}-E_{n+j,n+j})=\delta_{i, j}\end{equation*}

 For any two integers $p\leq q$, we denote $\ol{p,q}=\{p,p+1,\dots,q\}$. The canonical bosonic oscillator representation $\pi$ of $\mathfrak{o}(2n)$ on $\msr{A}$ is given as follows:
\begin{align*} \begin{split}
\pi(E_{i,j}-E_{n+j,n+i})=x_i\ptl_{x_j}-y_j\ptl_{y_i},
\pi(E_{n+i,j}-E_{n+j,i})=y_i\ptl_{x_j}-y_j\ptl_{x_i},\\
\pi(E_{i,n+j}-E_{ j,n+i})=x_i\ptl_{y_j}-x_j\ptl_{y_i}
\qquad\text{for}\;\;i,j\in  \ol{1,n}  .\end{split}\end{align*}

With respect to the above representation, the Laplace operator  is defined as
\begin{equation*}\Dlt=\sum_{i=1}^n\ptl_{x_i}\ptl_{y_i}.\end{equation*}
Denote by $\mbb N$ the set of nonnegative integers. For $\ell_1,\ell_2\in \mbb N$, we denote by $\msr A_{\ell_1,\ell_2}$ the subspace of homogeneous polynomials of degree $\ell_1$ in $\{x_1,x_2,...,x_n\}$ and the degree $\ell_2$ in $\{y_1,y_2,...,y_n\}$, and set
\begin{align*}
\msr A_{k}=\sum_{\ell=0}^k \msr A_{\ell,k-\ell},\,\, \msr H_{k} =\{f\in {\msr A}_{k}\mid
\Dlt(f)=0\} \qquad \mbox{for }\, k\in \mbb N.
\end{align*}
Then $\msr{H}_k$ is a finite-dimensional irreducible $\mathfrak{o}(2n)$-module, and $x_1^k$ is a highest-weight vector with highest weight $k \es_1$ ( cf.\cite{Xx}).

 Fix $n_1,n_2\in\ol{1,n}$ with $n_1\leq n_2$. Throughout the paper, put
\begin{equation}\label{eq:Jsets}
J_1=\ol{1,n_1},\qquad J_2=\ol{n_1+1,n_2},\qquad J_3=\ol{n_2+1,n}.
\end{equation}
An interval with its lower endpoint greater than its upper endpoint is empty.
Applying the Fourier transformations to the operators:
\begin{equation} \label{1.7}
\ptl_{x_r}\mapsto -x_r,\;
 x_r\mapsto\ptl_{x_r},\;\ptl_{y_s}\mapsto -y_s,\;
 y_s\mapsto\ptl_{y_s}\qquad \text{for}\; r\in\ol{1,n_1},\;s\in\ol{n_2+1,n},\end{equation}
All other multiplication and differentiation operators are fixed. We obtain a new representation $\td{\pi}$ on $\msr A$.
Under the changes in (\ref{1.7}), the Laplace operator becomes (up to a sign)
\begin{equation} \label{1.8}
\td\Dlt=\sum_{i=1}^{n_1}x_i\ptl_{y_i}-\sum_{r=n_1+1}^{n_2}\ptl_{x_r}\ptl_{y_r}+\sum_{s=n_2+1}^n
y_s\ptl_{x_s}. \end{equation}
In particular, $\td\Dlt= \sum\limits_{i=1}^{n_1}x_i\ptl_{y_i} +\sum\limits_{s=n_2+1}^n
y_s\ptl_{x_s}$ when $n_1=n_2$.

Take the degree of $\{x_1,...,x_{n_1},y_{n_2+1},...,y_n\}$ as $-1$ and that of $\{x_{n_1+1},...,x_n,y_1...,y_{n_2}\}$
as 1. Denote by $\mbb Z$ the ring of integers. For $\ell \in \mbb Z$, we denote by $\msr A_{\la\ell \ra}$ the subspace of homogeneous polynomials of degree $\ell $ in $\{x_1,x_2,...,x_n, y_1,y_2,...,y_n\}$.
 Each $\msr A_{\la\ell\ra}$ is $\widetilde\pi$-invariant. Then $\msr A=\bigoplus_{\ell \in\mbb Z}\msr A_{\la\ell \ra}$ is a $\mbb Z $-graded algebra. Define
\begin{equation*}{\msr H}_{\la\ell \ra}=\{f\in {\msr A}_{\la\ell \ra}\mid
\td{\Dlt}(f)=0\}.\end{equation*}Recall the following structural result from \cite{LX}: \pse

{\bf Proposition 1}\quad {\it  For any integer $\ell$ satisfying
$\ell \leq n_1-n_2+1-\delta_{n_1,n_2} $, ${\msr H}_{\la\ell \ra}$ is an infinite-dimensional irreducible $\mfk{o}(2n)$-weight module. The module ${\msr H}_{\la\ell \ra}$ under the assumption is of highest-weight type only if $n_2=n$. When $n_2=n$, $x_{n_1}^{-\ell}$ is a highest-weight vector with weight $\ell  \es_{n_1}- \sum_{i=1}^{n_1} \es_i$.  } \psp

Let $\mfk g$ be a finite-dimensional semisimple Lie algebra and let $U(\mfk g)$ be its universal enveloping algebra. Set $U_0(\mfk g)=\mbb F$ and, for a positive integer $m$, define $U_m(\mfk g)=\mbb F+\sum_{r=1}^m\mfk g^r$. Then $\{U_k(\mfk g)\mid k\in\mbb N\}$ is a filtration of $U(\mfk g)$. Suppose $M$ is a finitely generated $U(\mathfrak{g})$-module and $M_0$ is a finite-dimensional subspace of $M$ such that
\begin{equation}\label{1.10}
(U(\mfk g))(M_0)=M.\end{equation}
Setting $M_k =(U_k( \mathfrak{g}))(M_0)$, we obtain a filtration $\{M_k\mid k\in\mbb N\}$ of $M$.

The associated graded module is defined as
\begin{equation}\label{1.11}
\ol{M}=\mbox{gr}( M; M_0) = \bigoplus_{k = 0}^\infty \ol{M_k}\quad \text{with}\quad \ol{M}_k=M_k / M_{k-1},
\end{equation}
where $M_{-1}=0$. It is a finitely generated graded $S( \mathfrak{g})$-module, since the graded algebra $\mbox{gr}(  U( \mathfrak{g}))$ is canonically isomorphic to the symmetric tensor algebra $S( \mathfrak{g})$ by the PBW theorem. Here we identify $S( \mathfrak{g})$ with $\mathbb{F}[\mathfrak{g}^*]$, the polynomial ring over $\mathfrak{g}^*$, canonically; the Killing form is needed only to identify $\mathfrak g^*$ with $\mathfrak g$.

The annihilator of $\ol{M}$ in $S( \mathfrak{g})$ is
\begin{equation*}
\operatorname{Ann}_{S( \mfk g)} (\ol{M})=\{ \xi \in  S( \mathfrak{g})\mid \xi(v)=0,\; \forall\; v \in \ol{M} \},\end{equation*}
which is a graded ideal of $S( \mathfrak{g})$. It defines the associated variety in the dual space $\mathfrak{g}^*  $:
\begin{equation*}
{\msr V}(M)=\{ f \in \mathfrak{g}^*\mid u(f) =0,\;\forall\; u\in \operatorname{Ann}_{S(\mathfrak{g})}(\ol{M})\}.\end{equation*}
We use the reduced affine scheme
\[
\msr V(M)=\operatorname{Spec}\bigl(S(\mfk g)/\sqrt{\operatorname{Ann}_{S(\mfk g)}(\overline M)}\bigr).
\]
Its projectivization is denoted by $\mathbb P\msr V(M)$ and is the corresponding $\operatorname{Proj}$. These schemes are independent of the choice of $M_0$ satisfying \eqref{1.10}. All intersections and isomorphisms of varieties below are taken with their reduced scheme structures over $\mathbb F$. In particular, an affine cone is distinguished from its projectivization.

For $m \in \mbb{N}$, let $\mathrm{Skew}(m)$ denote the space of $m \times m$ skew-symmetric matrices over $\mbb{F}$. For $r\in\mbb N$, the Pfaffian variety is defined by
\begin{align*} 
{\bf Pf}(m, 2r) = \{ X \in \mathrm{Skew}(m) \mid \mathrm{rank}(X) \leq 2r \}.
\end{align*} 
For a subset $S \subseteq  \{1,2, \dots, 2n  \} \times \{  1,2, \dots, 2n  \} $, we define the coordinate subspace $\msr{V}_1(S)$ by the vanishing of the corresponding entries:
\begin{align*} 
\msr{V}_1(S) = \{ X = (X_{ij}) \in \mathrm{Skew}(2n) \mid X_{ij} = 0 \text{ for } (i, j) \in S \}.
\end{align*} 
For any subvariety $Y \subseteq \mathrm{Skew}(2n)$, we refer to $Y \cap \msr{V}_1(S)$ as a coordinate linear section of $Y$ (\cite{BV88,CW19}).  

For $r, m\in \mbb{N}$, we denote by $\mathbf{Gr}(r,m)$ the Grassmannian parameterizing the set of all $r$-dimensional linear subspaces of an $m$-dimensional vector space over $\mbb{F}$.

 Define the set
\begin{align}\label{eq:S0}\begin{split}
S_0=&(\ol{n_2+1,n+n_1}\times \ol{1,2n})\cup (\ol{n_1+1,n+n_2}\times \ol{n_1+1,n+n_2})\\
 &\cup (\ol{1,2n} \times \ol{n_2+1,n+n_1}).
\end{split}
\end{align} 

The following is our first main result.\pse

	{\bf Theorem 1}\quad {\it Let $\mathfrak{g} = \mathfrak{o}(2n)$ and let $\msr{H}_{\langle \ell \rangle}$ be defined as above. If $\ell \leq 0$,
\begin{align*}
\msr{V}(\msr{H}_{\la\ell \ra}) \cong  \msr V_1(S_0)  \cap {\bf Pf}(2n, 2).
\end{align*}
In particular, when $\ell \leq 0$, $n_1 = n_2$, and $n \geq 3$, we have
\begin{align*} 
\mathbb P\msr V(\msr H_{\la\ell\ra}) \cong \mathbf{Gr}(2,n)
\end{align*}
as projective varieties.}\psp

Theorem 1 provides a complete geometric characterization of the associated varieties for the irreducible $\mathfrak{o}(2n)$-modules $\msr{H}_{\langle \ell \rangle}$ introduced in Proposition 1. \pse

Now we consider the odd case. The split odd orthogonal Lie algebra
\begin{align*}
\mfk{o}(2n+1)=\mfk{o}(2n)+\sum_{i=1}^n[\mbb F(E_{0,i}-E_{n+i,0})+\mbb F(E_{0,n+i}-E_{i,0})].
 \end{align*}
Let $\msr{A}' = \mathbb{F}[x_0, x_1, \dots, x_n, y_1, \dots, y_n]$. The canonical oscillator representation $\pi'$ of $\mathfrak{o}(2n+1)$ on $\msr{A}'$ is given by
\begin{align} \label{b1.23}
\begin{split}
\pi'(E_{i,j}-E_{n+j,n+i})=x_i\ptl_{x_j}-y_j\ptl_{y_i},
\pi'(E_{n+i,j}-E_{n+j,i})=y_i\ptl_{x_j}-y_j\ptl_{x_i},\\
\pi'(E_{i,n+j}-E_{ j,n+i})=x_i\ptl_{y_j}-x_j\ptl_{y_i},
\pi'(E_{0,i}-E_{n+i,0})= x_0\ptl_{x_i}-y_i\ptl_{x_0},\\
\pi'(E_{0,n+i}-E_{i,0})= x_0\ptl_{y_i}-x_i\ptl_{x_0}
\qquad\text{for}\;\;i,j\in\ol{1,n}.\end{split}
 \end{align}
Applying the Fourier transformation defined by (\ref{1.7}) to the operators in (\ref{b1.23}), we obtain a new representation $\widetilde{\pi}'$ on $\msr{A}'$. The corresponding transformed Laplace operator is
\begin{equation*}\td\Dlt'=\ptl_{x_0}^2- 2 \sum_{i=1}^{n_1}x_i\ptl_{y_i} + 2 \sum_{r=n_1+1}^{n_2}\ptl_{x_r}\ptl_{y_r}- 2 \sum_{s=n_2+1}^n y_s\ptl_{x_s}. \end{equation*}
Set
\begin{equation*}{\msr A}_{\la\ell \ra}'=\sum_{i=0}^{\infty} {\msr A}_{\la\ell-i \ra} x_0^{i},\qquad {\msr H}_{\la\ell \ra}'=\{f\in {\msr A}_{\la\ell \ra}'\mid
\td{\Dlt}'(f)=0\}.\end{equation*}The following is a result from \cite{LX}: \pse

{\bf Proposition 2}\quad {\it  For any $\ell \in \mbb Z$, the space $\msr H_{\la \ell \ra}'$ is an infinite-dimensional irreducible $\mfk{o}(2n+1)$-weight module with finite-dimensional weight subspaces, and it is of highest-weight type only if $n_2 = n$. } \pse 
 
Set
\begin{align*} 
 S_0'=S_0 \cup (\{0\} \times \ol{n_1+1,n+n_2}) \cup (  \ol{n_1+1,n+n_2} \times \{0\} ) .\end{align*} 

The notation $\msr V_1(S)$ is used in $\mathrm{Skew}(2n+1)$ in the same way, with indices $0,\ldots,2n$. The following is our second main result. \pse

{\bf Theorem 2}\quad {\it For $\mfk g=\mathfrak{o}(2n+1)$ and ${\msr H}_{\la\ell \ra}'$ defined as above,
\begin{align*}
\begin{split}
\msr V({\msr H}'_{\la\ell \ra})\cong    \msr V_1(S_0') \cap {\bf Pf}(2n+1,2 ).
\end{split}\end{align*}   }

Thus, for fixed $n,n_1,n_2$, the associated variety in Theorem 2 is independent of $\ell$. The coordinate section still depends on $n_1,n_2$.

Sections 2--5 give the weighted degree lemmas, the filtration and the defining equations in the even case. Section 6 treats the odd case separately. The longer computations used in the filtration and the relation arguments are given in the appendices.

\section{Weighted Degree Functions and Filtrations}
\
Let $\mathscr{C} = \mathbb{F}[x_1, \ldots, x_n]$ be the polynomial algebra in $n$ variables over a field $\mathbb{F}$, and let $\mathfrak{g}$ be a finite-dimensional Lie algebra. Consider a $U(\mathfrak g)$-module $M$ contained in $\mathscr C$.

In this section, we introduce the concept of a weighted degree function. Under suitable conditions on $M$, we can use such a function to describe a filtration of the form $U_k(\mfk{g})(M_0)$, where $M_0$ is a finite-dimensional subspace satisfying \eqref{1.10} (Lemma~\ref{lem:2.1}). Lemma~\ref{lem:2.1} separates the filtration argument into a degree estimate and a lowering condition. Lemma~\ref{lem:2.2} records the strict increase of the maximal degree under multiplication. These observations explain how highest degree terms can be used to find the defining equations. The results of this section will be applied in the subsequent sections to both the $\mfk{o}(2n)$ and $\mfk{o}(2n+1)$ cases. \pse

For any multi-index $\alpha = (\alpha_1, \alpha_2, \ldots, \alpha_n) \in \mathbb{N}^n$, we write the corresponding monomial as
\begin{equation*}
x^{\alpha} = x_1^{\alpha_1} x_2^{\alpha_2} \cdots x_n^{\alpha_n}.
\end{equation*}

Let $\Gamma$ be a subgroup of $(\mathbb{R},+)$.
A function $\mathfrak{d} : \msr{C} \to \Gamma \cup \{ -\infty\}$ is called a {\it weighted degree function} (\cite{KS}) if it satisfies the following conditions for $f,g  \in \msr{C}$:
\begin{align}\label{2.2}\begin{split}
&(1)\;{\mfk d} \left(f\cdot g\right) ={\mfk d} (f) + {\mfk d} (g)\ ;    \\\
&(2)\;{\mfk d} \left(  0  \right) = -\infty ; \\
&(3)\;{\mfk d} \left( c \right) =0  \quad \text{for all } c \in \mathbb{F}\setminus\{0\}; \\
&(4)\;\mathfrak d\left(\sum_{i=1}^r f_i\right)=\max_{1\leq i\leq r}\mathfrak d(f_i),\\
& \;\;\text{where the $f_i$ are nonzero monomial terms with distinct exponents. }
\end{split}\end{align}
Equivalently, any weighted degree function $\mathfrak{d}$ is determined by a weight vector $(r_1, \dots, r_n) \in \Gamma^n$ such that for any monomial $\lambda x^{\alpha}$ with $\lambda \in \mathbb{F}\setminus \{0\}$ and multi-index $\alpha \in \mathbb{N}^n$,
\begin{align}\label{2.3}
\mathfrak{d}(\lambda x^{\alpha}) = \sum_{i=1}^n \alpha_i r_i,
\end{align}
and conditions (2), (4) above hold.

Fix a weighted degree $\mfk d_1$. For each integer $k \in \mathbb{Z}$, we define the subspace
\begin{align*}
 N(k )=\operatorname{Span}\big\{ x^\alpha \in \msr C \mid \,k-1<{\mfk d}_1(x^\alpha) \leq k \big\},
 \end{align*}
which is a subspace of $\mathscr{C}$.
Given $f \in \mathscr{C}$ and a weighted degree function $\mathfrak{d}$, a {\it $\mfk{d}$-$\mbb Z$-decomposition} of $f$ is an expression $f = \sum_{l \in \mathbb{Z}} f_l$, where each $f_l \in N(l)$.

Let $\{u_1,\dots,u_d\}$ be a basis of $\mathfrak{g}$. Define the index set
\begin{align}\label{2.4}
I_0=\big\{ i \in \ol{1,d} \mid \exists p_i\in\msr C\;\;\forall f\in M,\quad u_i(f)=f\cdot p_i \big\},
 \end{align}and let $P_0= \{ p_i\in \msr C \mid \;u_i(f)=f\cdot p_i ,\; i\in I_0,\;f\in M \}$  be the corresponding set of multipliers.

 Choose a subset $I_1 \subseteq I_0$, and define the corresponding set $P_1=\{ p_i\in \msr C \mid \; i\in I_1 \} $ (where each $p_i$ is taken from the family defining $P_0$).
Suppose a weighted degree function $\mathfrak{d}_1$ satisfies, for all $f \in M$:
\begin{align}  \label{2.5}
    \text{  $\mathfrak{d}_1(u_i(f)) \leq 1 + \mathfrak{d}_1(f)$ for all $i \notin I_1$}.
 \end{align}
Let $\msr C_0$ be a $\mfk d_1$-graded subspace of $\msr C$, stable under multiplication by $P_1$, and let $T:\msr C_0\to M$ be a degree-preserving linear isomorphism  satisfying the following conditions for every $f \in \msr C_0$:
\begin{align}
\label{2.6}   & \text{(1) $T(N(l)\cap  \msr C_0) =N(l) \cap M$\; for each $l\in \mbb Z$;}\\
   &   \text{(2) $T(f\cdot p)=T(f)\cdot p$\;  for all $p\in P_1$,}\label{2.7}
 \end{align}

Fix $k_0\in \mbb Z$. For $k\in\mathbb Z$, define $TN_{k,k_0}$ and $TN(k,k_0)$ below; define $V_{k,k_0}$ only for $k\geq0$:
\begin{align}
TN_{k,k_0}&=\operatorname{Span}\big\{T(f)  \mid f \in \msr C_0 ,\,{\mfk d}_1(f)\leq k +k_0 \big\},\notag\\
TN(k,k_0)&=\operatorname{Span}\big\{T(f)  \mid f=\sum_i f_i \in \msr C_0 \text{, where $f_i$ are monomials},\notag\\
&\qquad\qquad\,k +k_0-1<{\mfk d}_1(f_i)\leq k +k_0 \text{ for each } i\big\},\notag\\
 \label{2.11}
V_{k,k_0}&=\operatorname{Span}\big\{f_i \cdot g_i \mid f_i\in TN_{i,k_0},\,g_i \in P_1^{k-i},\,i \in\ol{0,k}  \big\},
 \end{align}where $P_1^k$ denotes the set of all products of $k$ elements from $P_1$ (with the convention $P_1^0 = \{1\}$).
Then, for any $r \in \ol{0,k}$, we have the inclusion\begin{align}\label{2.12}
 V_{r,k_0+k-r}\subseteq V_{k,k_0}  .
 \end{align}

We have the following lemma.
\begin{lemma}\label{lem:2.1}
For all $k_0\in \mbb Z$ and $k \in \mathbb{N}$, we have $U_k(\mathfrak{g}) ( V_{0,k_0} )\subseteq V_{k,k_0}$. Moreover, equality holds for all $k \in \mathbb{N}$ (i.e., $U_k(\mathfrak{g}) ( V_{0,k_0} ) = V_{k,k_0}$) if and only if
\begin{align} \label{2.13}
TN(k,k_0)  \subseteq \sum_{i=1}^k U_{i }(\mathfrak{g}) ( V_{k-i ,k_0} )
 \end{align}
 for all $k \geq1$.
\end{lemma}

\begin{proof}
We prove $U_k(\mathfrak{g})(V_{0,k_0}) \subseteq V_{k,k_0}$ for any $k \in \mathbb{N}$ by induction on $k$. The case $k=0$ is trivial.

Assume inductively that $U_i(\mathfrak{g}) ( V_{0,k_0} )\subseteq V_{i,k_0}$ for all $i\leq k-1$ and every integer $k_0$. Consider the step $k$. Observe that
\begin{align*}
 U_k\left(\mfk g \right)(V_{0,k_0})=(\mathbb F+\mfk g)(U_{k-1}(\mfk g)(V_{0,k_0}))\subseteq V_{k-1,k_0}+\mfk g(V_{k-1,k_0}).
\end{align*}
Let $ f \in TN_{k-1,k_0}$.
For $i  \in  \ol{1,d} \setminus I_1$, Inequality (\ref{2.5}) gives ${\mfk d}_1 \left( u_i(f) \right)\leq {\mfk d}_1 \left( f \right)+1\leq k+k_0$.
By (\ref{2.6}) and  the fact that  \begin{align*} u_i(f) \in M=\bigcup_{l=0}^{\infty}TN_{l,k_0},\end{align*} we obtain  $u_i(f) \in TN_{k,k_0}$.
For $i\in  I_1$, (\ref{2.4}) and (\ref{2.11}) yield
\begin{equation*}
u_i(f) \in TN_{k-1,k_0} \cdot P_1 \subset V_{k,k_0}.\end{equation*}

Now fix an element $g \in TN_{k-r-1,k_0}\cdot  P_1^r$ for some $r\in\ol{1,k-1}$.
If $i \in I_1$, we have
\begin{equation*}u_i(g)\in TN_{k-r-1,k_0}\cdot P_1^{r+1}  \subseteq V_{k,k_0}.\end{equation*}
Now suppose $i\in \ol{1,d} \setminus I_1$.
By linearity and (\ref{2.4}), it suffices to consider $g$ of the form
\begin{equation*}g=u_{i_1}\cdots u_{i_r }(g')\end{equation*}
where $g'\in TN_{k-r-1,k_0}\subset V_{k-r-1,k_0 }$ and $ i_1, ..., i_r \in I_1$.
 Applying $u_i$ and commuting $u_i$ past the factors, we obtain
 \begin{eqnarray*}
u_i(g)&=&\sum_{s=1}^r (u_{i_1 }\cdots u_{i_{s-1} }[u_{i},u_{i_s}]
  u_{i_{s+1} }\cdots  u_{i_r }(g')) +u_{i_1}\cdots u_{i_r}u_{i }(g')\\ &\in& U_{r }(\mathfrak{g}) ( TN_{k-r-1,k_0})+ u_i(TN_{k-r-1,k_0})\cdot P_1^r
\\ &\in& U_{r }(\mathfrak{g}) ( V_{0,k_0+k-r-1})+ TN_{k-r,k_0}\cdot P_1^r
\\ &\in&   V_{r ,k_0+k-r-1} + TN_{k-r,k_0}\cdot P_1^r
\\ &\in&   V_{k-1,k_0  } + TN_{k-r,k_0}\cdot P_1^r
\subseteq V_{k,k_0}\end{eqnarray*}
by (\ref{2.12}).
Thus, in all cases we have $u_i(g) \in V_{k,k_0}$, which completes the proof of the inclusion $ U_k(\mathfrak{g}) ( V_{0,k_0} )\subset V_{k,k_0}$.
\pse

Assume that \begin{align*}
TN(k,k_0)  \subseteq \sum_{i=1}^k U_{i }(\mathfrak{g}) ( V_{k-i ,k_0} )
 \end{align*} holds for all integers  $k\geq 1$. We prove $U_k(\mathfrak{g}) ( V_{0,k_0} )= V_{k,k_0}$ by induction on $k$.
For the inductive step, suppose $U_i(\mathfrak{g}) ( V_{0,k_0} )= V_{i,k_0}$ for all $i \leq k-1$.
Observe first that
\begin{align*}
V_{k,k_0}
&=\sum_{i=0}^k TN_{k-i,k_0}\cdot P_1^i\\
&=\sum_{i=0}^k\left(TN_{0,k_0}\cdot P_1^i+
\sum_{j=1}^{k-i}TN(j,k_0)\cdot P_1^i\right).
\end{align*}
The first term in each parenthesis is contained in $U_i(\mfk g)(V_{0,k_0})$. For $1\leq j\leq k-i$, the lowering condition gives
\begin{align*}
TN(j,k_0)\cdot P_1^i
&\subseteq\sum_{s=1}^{j}U_i(\mfk g)U_s(\mfk g)(V_{j-s,k_0})\\
&\subseteq U_{i+j}(\mfk g)(V_{0,k_0})
\subseteq U_k(\mfk g)(V_{0,k_0}),
\end{align*}
where the second inclusion follows from the inductive hypothesis, since $j-s\leq k-1$. Hence $V_{k,k_0}\subseteq U_k(\mfk g)(V_{0,k_0})$, proving the reverse inclusion.

Conversely, if $V_{k,k_0}=U_k(\mfk g)(V_{0,k_0})$ for every $k\geq0$, the term $i=k$ on the right of \eqref{2.13} contains $TN(k,k_0)$ for $k\geq1$.
This completes the proof of the equivalence stated in Lemma \ref{lem:2.1}.

\end{proof}
 \psp

In the rest of this section, suppose that $U_k(\mfk g)(V_{0,k_0})=V_{k,k_0}\neq\{0\}$ for every $k\geq0$, that $V_{0,k_0}$ is finite-dimensional, and that $I_1\neq\emptyset$.

For a polynomial $f=\sum_i f_i$, where the $f_i$ are its nonzero monomial terms, we say that $f$ is {\it $\mfk d$-homogeneous of degree $c$} if $\mfk d(f_i)=c$ for every $i$. Suppose that the nonzero elements of $P_1$ are $\mfk d_1$-homogeneous of degree $m_1\geq1$ and $\mfk d_2$-homogeneous of degree $m_2$, for another weighted degree function $\mfk d_2$, and that $P_1$ contains a nonzero element. In the notation $d=\mfk d_1+\mfk d_2$, the weights on the variables are added and then extended to polynomials by taking a maximum. Set
\[
m_{(k)}=\max\{d(f):0\neq f\in V_{k,k_0}\}.
\]
This maximum exists since $V_{k,k_0}=U_k(\mfk g)(V_{0,k_0})$ is finite-dimensional.

\begin{lemma}\label{lem:2.2}
If $m_1+m_2\geq1$, then, for every integer $k\geq1$,
\[
m_{(k-1)}<m_{(k)}.
\]
\end{lemma}
\begin{proof}
Fix $i_1\in I_1$ with $p_{i_1}\neq0$, and choose $F\in V_{k-1,k_0}$ such that $d(F)=m_{(k-1)}$. By \eqref{2.11},
\[
u_{i_1}(F)=p_{i_1}F\in V_{k,k_0}.
\]
The product is nonzero, and $d(p_{i_1})=m_1+m_2$. Thus
\begin{align*}
m_{(k)}&\geq d\bigl(u_{i_1}(F)\bigr)
=d(p_{i_1}F)\\
&=m_1+m_2+m_{(k-1)}>m_{(k-1)}.
\end{align*}
\end{proof}
The proof uses the multiplication operator to increase the degree of an element of maximal degree in the preceding filtration space. This strict increase is used when comparing highest degree terms in Section~5.

\section[The Filtration Structure of the Oscillator Representation]{The Filtration Structure of the Oscillator Representation $\widetilde{\pi}$}

In this section, $M=\msr H_{\la\ell\ra}$. We choose $M_0$ and apply Lemma~\ref{lem:2.1} to describe $M_k$. The weighted degree measures the increase caused by each oscillator operator.

In Sections~3--5, we consider only the noncanonical oscillator representation $\widetilde{\pi}$ of $\mathfrak{o}(2n)$ obtained via the Fourier transformations. Its explicit formulas are given by
\begin{equation} \label{3.1}
\td{\pi}(E_{i,j}-E_{n+j,n+i})=E_{i,j}^x-E_{j,i}^y\qquad\text{for}\;\;i,j\in\ol{1,n}\end{equation} where
\begin{equation*}E_{i,j}^x=\left\{\begin{array}{ll}-x_j\ptl_{x_i}-\delta_{i, j}&\mbox{if}\;
i,j\in\ol{1,n_1};\\ \ptl_{x_i}\ptl_{x_j}&\mbox{if}\;i\in\ol{1,n_1},\;j\in\ol{n_1+1,n};\\
-x_ix_j &\mbox{if}\;i\in\ol{n_1+1,n},\;j\in\ol{1,n_1};\\
x_i\partial_{x_j}&\mbox{if}\;i,j\in\ol{n_1+1,n};
\end{array}\right.\end{equation*}
and
\begin{equation*}
E_{i,j}^y=\left\{\begin{array}{ll}y_i\ptl_{y_j}&\mbox{if}\;
i,j\in\ol{1,n_2};\\ -y_iy_j&\mbox{if}\;i\in\ol{1,n_2},\;j\in\ol{n_2+1,n};\\
\ptl_{y_i}\ptl_{y_j} &\mbox{if}\;i\in\ol{n_2+1,n},\;j\in\ol{1,n_2};\\
-y_j\partial_{y_i}-\delta_{i, j}&\mbox{if}\;i,j\in\ol{n_2+1,n}.
\end{array}\right.\end{equation*}
For the next two formulas only, $\widetilde\pi(E_{r,s})$ denotes the Fourier transform of the ambient operator $z_r\partial_{z_s}$, where $(z_1,\ldots,z_{2n})=(x_1,\ldots,x_n,y_1,\ldots,y_n)$. The orthogonal generators act by the corresponding differences.
\begin{equation*}
\td{\pi}(E_{i,n+j}) =\left\{\begin{array}{ll}\ptl_{x_i}\ptl_{y_j}&\mbox{if}\;
i \in \ol{1,n_1},j\in\ol{1,n_2};\\ - y_j \ptl_{x_i} &\mbox{if}\;i\in\ol{1,n_1},\;j\in\ol{n_2+1,n};\\
x_i\ptl_{y_j} &\mbox{if}\;i\in\ol{n_1+1,n},\;j\in\ol{1,n_2};\\
-x_iy_j &\mbox{if}\;i\in\ol{n_1+1,n},j\in\ol{n_2+1,n};
\end{array}\right.\end{equation*}
and
\begin{equation} \label{3.5}
\td{\pi}(E_{n+i,j}) =\left\{\begin{array}{ll}
-x_jy_i&\mbox{if }j\leq n_1,\ i\leq n_2;\\
-x_j\partial_{y_i}&\mbox{if }j\leq n_1,\ i>n_2;\\
y_i\partial_{x_j}&\mbox{if }j>n_1,\ i\leq n_2;\\
\partial_{x_j}\partial_{y_i}&\mbox{if }j>n_1,\ i>n_2.
\end{array}\right.\end{equation}

 Denote
\begin{equation*}x^\al=x_1^{\al_1}x_2^{\al_2}\cdots x_n^{\al_n},\;\;y^\al=y_1^{\al_1}y_2^{\al_2}\cdots y_n^{\al_n}\qquad \text{for}\;\;\al=(\al_1,\al_2,...,\al_n)\in \mbb N^n.\end{equation*}
Moreover, we adopt the notation
\begin{equation*}\xi(f)=\td{\pi}(\xi)(f)\qquad\text{for}\;\;\xi\in \mfk{o}(2n),\;f\in\msr A.\end{equation*}
For $\ell \in \mbb Z$, set
\begin{align*}\begin{split}
\msr A_{\la \ell\ra} =\operatorname{Span} \{x^\al y^\be \mid
\al,\be \in \mbb N^n;\, \sum_{r=n_1+1}^n \al_r - \sum_{i= 1}^{n_1} \al_i + \sum_{i= 1}^{n_2} \be_i
-\sum_{r= n_2+1 }^{n } \be_r = \ell  \}.\end{split}\end{align*}

In the rest of this section, we keep the assumptions $\ell \leq 0$ and $n_1 < n_2$. The irreducibility range in Proposition 1 is contained in $\ell\leq0$. Irreducibility is not assumed in the filtration argument.

Recall the transformed Laplace operator $\widetilde\Delta$ in \eqref{1.8}. Define the linear operator $T$ on $\msr A$ by
\begin{align*}
T(x^\al y^\be)=\sum_{i=0}^{\infty} \frac{\left(x_{n_1+1} y_{n_1+1}\right)^i
	 \left(\tilde{\Delta}+\partial_{x_{n_1+1}}\partial_{ y_{n_1+1}}\right)^i }
 {\prod_{r=1}^i \left(\alpha_{n_1+1} +r\right) \left(\beta_{n_1+1}+ r\right)}(x^\al y^\be)
\end{align*}
 for $\al,\be\in \mbb N^n$.
 Put $h=n_1+1$. Each term of $\widetilde\Delta+\partial_{x_h}\partial_{y_h}$ avoids $x_h,y_h$ and lowers the total degree in the remaining positive-weight variables. Thus the sum defining $T$ is finite on every monomial. Since the orthogonal operators on the other indices commute with both this operator and multiplication and integration in $x_h,y_h$, we have
\begin{align}\label{3.10}
T\widetilde\pi(A)=\widetilde\pi(A)T,\qquad
A\in\{E_{i,j}-E_{n+j,n+i},E_{i,n+j}-E_{j,n+i},E_{n+i,j}-E_{n+j,i}\},
\quad i,j\neq h.
\end{align}

 Moreover, the set
\begin{align*}
\left\{ T\left(x^\alpha y^\beta \right) \mid \alpha,\beta  \in \mathbb{N}^n,\alpha_{n_1+1}\beta_{n_1+1}=0,x^\alpha y^\beta \in {\msr A}_{\la \ell \ra}\right\}
\end{align*}forms a basis of the module ${\msr H}_{\la \ell \ra}$, as shown in \cite{LX}. For completeness, let $Q(x^\alpha y^\beta)=x_hy_hx^\alpha y^\beta/((\alpha_h+1)(\beta_h+1))$, and let $\operatorname{pr}_0$ delete monomials divisible by $x_hy_h$. If $R=\widetilde\Delta+\partial_{x_h}\partial_{y_h}$, then $T=\sum_{i\geq0}(QR)^i$. On the subspace with $\alpha_h\beta_h=0$, one has $\operatorname{pr}_0T=1$, proving injectivity. A harmonic polynomial $g$ satisfies $g=\operatorname{pr}_0g+QRg$; iteration gives $g=T(\operatorname{pr}_0g)$. This also proves surjectivity onto the harmonic space.

Define a weighted degree function  ${\mfk d}_1 : {\msr A}  \rightarrow \mathbb{Z}\cup \{-\infty\}$   (cf.(\ref{2.2})-(\ref{2.3})) by
\begin{align*} 
{\mfk d}_1 (x^\al y^\be)= 2 \sum_{i \in J_3} \alpha_i +\sum_{i \in J_2}\alpha_i +2\sum_{i \in J_1}\beta_i+\sum_{i \in J_2}\beta_i.
\end{align*}

The factor $x_{n_1+1}y_{n_1+1}$ has weight $2$ and compensates for the weight $-2$ of each term of $R$. For $\al,\be\in\mbb N^n$ with  $\alpha_{n_1+1} \beta_{n_1+1} =0$ and $ \left(x_{n_1+1} y_{n_1+1}\right)  \left(\tilde{\Delta}+\partial_{x_{n_1+1}}\partial_{ y_{n_1+1}}\right)  \left( x^\alpha y^\beta \right) \neq 0$, we have
\begin{eqnarray*}&
&{\mfk d}_1( (x_{n_1+1} y_{n_1+1})
(\tilde{\Dlt}+\partial_{x_{n_1+1}}\partial_{ y_{n_1+1}})( x^\alpha y^\beta)) \\
&=&{\mfk d}_1((x_{n_1+1} y_{n_1+1})
( \sum^{n_1}_{i=1} x_i \partial_{y_i}-\sum^{n_2}_{r=n_1+2} \partial_{x_r} \partial_{y_r}+\sum^{n}_{s=n_2+1} y_s \partial_{x_s}
)( x^\alpha y^\beta))\\
&=&\max \{{\mfk d}_1( x_{n_1+1} y_{n_1+1} x_i \partial_{y_i}( x^\alpha y^\beta )),\;
{\mfk d}_1(  x_{n_1+1} y_{n_1+1}\partial_{x_r} \partial_{y_r}( x^\alpha y^\beta)),
\\\pse
 && {\mfk d}_1(  x_{n_1+1} y_{n_1+1}  y_s \partial_{x_s}( x^\alpha y^\beta))\mid
i\in J_1, n_1+2 \leq r \leq n_2, s \in J_3\}\\\pse
&=&{\mfk d}_1 \left(x^\alpha y^\beta\right).
\end{eqnarray*}
In general,
\begin{align*}
{\mfk d}_1 ((x_{n_1+1} y_{n_1+1})^i
(\tilde{\Delta}+\partial_{x_{n_1+1}}\partial_{ y_{n_1+1}})^i ( x^\alpha y^\beta))
={\mfk d}_1 (x^\alpha y^\beta) \text{ or } -\infty,
\end{align*}
where $\alpha_{n_1+1} \beta_{n_1+1} =0$.
Every nonzero summand is homogeneous of this weight, and the projection of their sum under $\operatorname{pr}_0$ is the original monomial. Thus there is no total cancellation, and we obtain
\begin{align*}
{\mfk d}_1 (T( x^\alpha y^\beta)) ={\mfk d}_1(x^\alpha y^\beta),\quad\text{ where $ \alpha_{n_1+1} \beta_{n_1+1} =0$.}
\end{align*}

Recall the sets $J_1,J_2,J_3$ in \eqref{eq:Jsets}. Then we have
\begin{align*}
{\mfk d}_1 \left( (E_{i,j}-E_{n+j,n+i})(f) \right)  \leq
\begin{cases}
{\mfk d}_1 \left( f \right)+2  &\text{if} \left( i,j\right)  \in J_3 \times J_1, \\
{\mfk d}_1 \left( f \right)+1  &\text{if}  \left( i,j\right)  \in   J_2 \times J_1  \cup  J_3 \times J_2,  \\
{\mfk d}_1 \left( f \right) &\text{if}  \left( i,j\right)  \notin  J_2 \times J_1  \cup  J_3 \times J_2 \cup J_3 \times J_1 ;
\end{cases}
\end{align*}
\begin{align*}
{\mfk d}_1 \left( (E_{i,n+j}-E_{j,n+i})(f) \right)  \leq
\begin{cases}
{\mfk d}_1 \left( f \right)+2  &\text{if} \left( i,j\right)  \in J_3 \times J_3, \\
{\mfk d}_1 \left( f \right)+1  &\text{if}  \left( i,j\right)  \in   J_2 \times J_3  \cup  J_3 \times J_2,  \\
{\mfk d}_1 \left( f \right) &\text{if}  \left( i,j\right)  \notin  J_2 \times J_3  \cup  J_3 \times J_2 \cup J_3 \times J_3 ;
\end{cases}
\end{align*}
and
\begin{align*}
{\mfk d}_1 \left( (E_{n+i,j}-E_{n+j,i})(f) \right)  \leq
\begin{cases}
{\mfk d}_1 \left( f \right)+2  &\text{if} \left( i,j\right)  \in J_1 \times J_1, \\
{\mfk d}_1 \left( f \right)+1  &\text{if}  \left( i,j\right)  \in   J_2 \times J_1  \cup  J_1 \times J_2,  \\
{\mfk d}_1 \left( f \right) &\text{if}  \left( i,j\right)  \notin  J_2 \times J_1  \cup  J_1 \times J_2 \cup J_1 \times J_1.
\end{cases}
\end{align*}

We set
\begin{equation*}
TN(k) = \operatorname{Span} \{  T (x^\alpha y^\beta )      \mid
{\mfk d}_1 \left(x^\alpha y^\beta  \right) = k ,x^\alpha y^\beta   \in {\msr A}_{\left \langle \ell \right \rangle  }, \alpha_{n_1+1} \beta_{n_1+1} =0  \},\end{equation*}
\begin{equation*}
TN_k = \operatorname{Span} \{  T (x^\alpha y^\beta )      \mid
{\mfk d}_1 \left(x^\alpha y^\beta  \right)\leq k ,x^\alpha y^\beta   \in {\msr A}_{\left \langle \ell  \right \rangle  }, \alpha_{n_1+1} \beta_{n_1+1} =0  \}.\end{equation*}

Let
\begin{align} \label{3.22}
W_1 = \{x_{i_1} y_{j_1}-x_{j_1} y_{i_1} \mid i_1,j_1 \in J_1, i_1<j_1\},
\end{align}\begin{align*}
W_3 = \{x_{i_3} y_{j_3}-x_{j_3} y_{i_3} \mid i_3,j_3 \in J_3, i_3<j_3\},
\end{align*}
\begin{align} \label{3.24}
P = \{-x_{i_1} x_{i_3}+y_{i_1} y_{i_3} \mid i_1 \in J_1, i_3\in J_3 \},
\end{align}
and $W=W_1\cup W_3 \cup P$.
Since \begin{align*} \begin{split}
(E_{i_3,i_1}-E_{n+i_1,n+i_3})(f)= (-x_{i_1} x_{i_3}+y_{i_1} y_{i_3})\cdot f,\\
(E_{i_3,n+j_3}-E_{j_3,n+i_3})(f)= (-x_{i_3} y_{j_3}+x_{j_3} y_{i_3})\cdot f,\\
(E_{n+i_1, j_1}-E_{n+j_1, i_1})(f)= (x_{i_1} y_{j_1}-x_{j_1} y_{i_1})\cdot f,
\end{split}\end{align*}
for $f \in \msr A$, we have \begin{align*}
T(f \cdot g) = T(f) \cdot g, \quad \text{for } f \in \msr A, g \in W  \end{align*} by (\ref{3.10}).

For $r,k\in\mbb N$, we set
\begin{align*}
TN(r) \cdot W^k =\operatorname{Span} \{fg\mid f\in TN(r),g\in W^k \}
\end{align*}
and
\begin{align*}
V_k&=\operatorname{Span}\{
TN(i),TN(k-i)W^i\mid i=0,1,\ldots,k\}\\
&=\operatorname{Span}\{
TN_k,TN(k-i)W^i\mid i=1,\ldots,k\}\\
&=\sum_{i=0}^k TN_{k-i}W^i\qquad(k\geq0).
\end{align*}
In particular, $V_0=TN_0$. For the last equality, use $TN(j)W\subseteq TN(j+2)$, which follows from the weight of $W$ and the commutation of $T$ with multiplication by $W$. If $j\leq k-i$, a term in $TN(j)W^i$ either has weight $j+2i\leq k$ and belongs to $TN_k$, or can be written in $TN(k-i')W^{i'}$ with $i'=j+2i-k$ by absorbing $i-i'$ factors into the first factor.

\begin{proposition}\label{prop:even-filtration} For $k\in \mbb N$, take $M_0=V_0$. Then $M_k=U_k(\mfk g)(V_0)=V_k$. Moreover, $M= {\msr H}_{\la\ell \ra}=U(\mfk g)(V_0)$.
\end{proposition}

\begin{proof}

Take $\msr C=\msr A$, $\msr C_0=\text{Span}\left\{ x^\alpha y^\beta   \in {\msr A}_{\la \ell \ra}  \mid \alpha_{n_1+1}\beta_{n_1+1}=0 \right\}$, $\Gamma=\mbb Z$, $k_0=0$ and $P_1=W$ in Lemma \ref{lem:2.1}. It can be verified that ${\mfk d}_1$, $T$ and $V_{k,k_0}=V_k$ satisfy the conditions of  \ref{lem:2.1}. Hence, by Lemma \ref{lem:2.1}, we have $M_k \subseteq V_k$ for all $k \in \mbb N$.

Here $V_0=TN_0$ is the space of monomials in $x_i$ ($i\in J_1$) and $y_j$ ($j\in J_3$) of degree $-\ell$, hence is finite-dimensional. Next we prove $V_k\subseteq M_k$ by induction on $k$, starting from $k=0$. By Lemma~\ref{lem:2.1}, it suffices to show that
$TN(k)\subseteq\sum_{i=1}^kU_i(\mfk g)(V_{k-i})$ for every $k\geq1$.

Define  two weighted degree functions ${\mfk l}_1,{\mfk l}_2 : {\msr A}  \rightarrow \mathbb{Z}\cup \{-\infty\}$ (cf.(\ref{2.2})-(\ref{2.3})) by
\begin{align*}
{\mfk l}_1 (x^\al y^\be)=  \sum_{i \in J_2\cup J_3} \alpha_i-\sum_{i \in J_1}\alpha_i ,\quad
{\mfk l}_2 (x^\al y^\be)=  \sum_{i \in J_1\cup J_2} \be_i-\sum_{i \in J_3}\be_i.
\end{align*}
Then $(\mfk l_1 + \mfk l_2)(x^\alpha y^\beta) = \ell$ for every $x^\alpha y^\beta \in M$.
For $f=T(f_0)=T(x^\al y^\be) \in TN(k)$, we distinguish two cases.
\pse

\medskip\noindent
\textbf{Case 1.} ${\mfk l}_1 (f_0),{\mfk l}_2 (f_0) \leq k-1$.\pse

Let $\mfk a= \operatorname{Span} \{E_{i,j}-E_{n+j,n+i} \mid  i,j\in\ol{1,n}\}$ be a Lie subalgebra of $\mfk o(2n)$.
Set $n_0=\max \{ {\mfk l}_1 (f_0),{\mfk l}_2 (f_0),0\}$ and for $t \in \mbb N$, we define
\begin{align*}\begin{split}
\ol{V}_t=\operatorname{Span}\big\{ &T(x^\al y^\be)\in TN_{n_0+t}   \mid  \al_{n_1+1}\be_{n_1+1} =0, \\ &{\mfk l}_1 (x^\al y^\be)={\mfk l}_1 (f_0), {\mfk l}_2 (x^\al y^\be)={\mfk l}_2 (f_0) \big\}.
\end{split}\end{align*}
On this fixed bidegree the center of $\mfk a$ acts by a scalar, so the $\mathfrak{sl}(n)$ filtration result applies to $\mfk a$ as well. According to Proposition 2.1 in \cite{ZX1}, we have
\begin{align*}
\ol{V}_t  \subseteq U_t(\mfk a) (\ol{V}_0 ) \qquad \text{for } t \in \mbb N.
\end{align*}
Consequently, $f\in U_{k-n_0}(\mfk a) (\ol{V}_0 ) \subseteq U_{k-n_0}(\mfk a) (V_{n_0}) \subseteq U_{k }(\mfk g) (V_{0})$.

\medskip\noindent
\textbf{Case 2.} ${\mfk l}_1 (f_0)=k$ or ${\mfk l}_2 (f_0)=k$.\pse

If ${\mfk l}_1 (f_0)=k$, we can write \begin{align*}  f= (\prod_{s=1}^k x_{i_s})( \prod_{t=1}^{k-\ell} y_{j_t} )=(x_{i_1}y_{j_1})g \end{align*} with $g=(\prod_{s=2}^k x_{i_s})( \prod_{t=2}^{k-\ell} y_{j_t} )$, where $i_s \in J_2$ and $j_t \in J_3$ for $s\in \ol{1,k}$ and $ t \in \ol{1,k-\ell}$.
Since $g\in V_{k-1}$, we have
\begin{align*}   f= (x_{i_1}y_{j_1})g =(E_{i_1,n+j_1}-E_{j_1,n+i_1})(-g) \in \mfk g (V_{k-1}).\end{align*}

 Symmetrically, if $\mfk l_2(f_0)=k$ (so $\mfk l_1(f_0)=\ell-k$), the same argument yields
$f \in \mfk g(V_{k-1})$.

\medskip\noindent
Hence $TN(k) \subseteq \sum_{i=1}^k U_i(\mfk g)(V_{k-i})$ for every
$k \in \mbb N$ with $k \ge 1$. By Lemma~\ref{lem:2.1}, we conclude that $M_k = V_k$ for all
$k \in \mbb N$.

Moreover,
 \begin{equation*}M=\msr H_{\la\ell \ra}\subset\bigcup_{k=0}^\infty TN_k\subset \bigcup_{k=0}^\infty V_k=\bigcup_{k=0}^\infty M_k=M,\end{equation*}
so that  (\ref{1.10}) holds.
\psp

\end{proof}

\section{Linear Independence of Alternating Polynomial Products}

We study the relations between the alternating polynomial products occurring in the filtration. We retain the chain notation of \cite{ZX2}. The defining ideal is described in Lemma~\ref{lem:4.4}.

Set
\begin{equation*}O_1=\{(j,i) \in J_1\times J_1\vert j>i\} \end{equation*} and
\begin{equation*}\Omega_1 = \{\omega_{j,i} \mid (j,i)\in O_1 \} \end{equation*}
so that $\Omega_1$ consists of $\frac{n_1(n_1-1)}{2}$ variables.
Similarly, we set \begin{equation*}O_3=\{(j,i) \in J_3\times J_3\vert j<i\} \end{equation*} and
\begin{equation*}\Omega_3 = \{\omega_{j,i} \mid (j,i)\in O_3 \} .\end{equation*}
For $m\in \mbb N$ and $i=1,3$, let $\Omega_i^{(m)}$ be the set of (monic) monomials of degree $m$ in the elements of $\Omega_i$.

Define
\begin{align*}
J_1'=\ol{0,n_1},\quad J_3'=\ol{n_2+1,n+1}
\end{align*}
and
 \begin{equation}z_{n+1,0}=0,\;z_{n+1,i}=x_i,\;z_{j,0}=y_j\qquad\text{for}\;\;i\in J_1,\;j \in J_3.
\label{a3.17} \end{equation}
Let
\begin{equation*}
\msr F=\mbb F[z_{j,i},\Omega_1,\Omega_3 \mid j\in J_3',\;i\in J_1',\;(j,i)\neq (n+1,0)]\end{equation*}
(cf.(\ref{a3.17})) be the algebra of polynomials in the variables $\{z_{j,i},x_i,y_j,\omega_{j',i'} \mid j\in J_3,\;i\in J_1,\;(j,i)\neq (n+1,0),(j',i')\in O_1\cup O_3\}$. Denote by $\msr F_k$ the subspace of homogeneous polynomials of degree $k$ in $\msr{F}$.

Define the associative algebra homomorphism $  \phi : \msr F \rta \msr A$ by
\begin{align*}
\phi(z_{j,i}) &=
\begin{cases}
y_i y_j - x_i x_j, & \text{if } j \in J_3,\; i \in J_1, \\[2pt]
x_i,                & \text{if } j = n+1,\; i \in J_1, \\[2pt]
y_j,                & \text{if } i = 0,\; j \in J_3,
\end{cases} \\
\phi(\omega_{j,i}) &= x_j y_i - x_i y_j, \qquad \text{for } (j,i) \in O_1 \cup O_3.
\end{align*}

For $N_1, N_2,N_3 \in \mathbb{N}$, $a \in \mathbb{F}$,
$ \{(u_r, v_r) \in J_3'\times J_1'\mid 1 \le r \le N_1\}$,
$\{(j_s, i_s) \in O_1  \mid 1 \le s \le N_2\}$  and
$\{(j'_t, i'_t) \in O_3  \mid 1 \le t \le N_3\}$, we set
\begin{align*}\begin{split}
&\msr I\left( a \prod_{r =1}^{N_1} z_{u_{r},v_{r }}   \prod_{s =1}^{N_2}\omega_{j_s,i_s}\prod_{t =1}^{N_3}\omega_{j'_t,i'_t} \right)  \\
=&
\{ (u_{r},v_{r}),(n+n_1+2-j_{s },i_{s }),(i'_{t},n_2-j'_{t}) \,
\mid\,    r \in \ol{1, N_1},  s \in \ol{1, N_2},   t\in \ol{1, N_3} \}.\end{split}
 \end{align*}

Let $\{(j_1,i_1),(j_2,i_2),\dots,(j_r,i_r)\} \subset \mathbb{Z}^2$.
We say that this set contains a {\it 3-chain} $(j_{k_1},i_{k_1}) \prec (j_{k_2},i_{k_2}) \prec (j_{k_3},i_{k_3})$ if there exist three distinct indices $1 \le k_1 , k_2 , k_3 \le r$ such that
$i_{k_1} < i_{k_2} < i_{k_3}$ and $j_{k_1} < j_{k_2} < j_{k_3}$.
Similarly, it contains a \emph{2-chain} $(j_{k_1},i_{k_1}) \prec (j_{k_2},i_{k_2})$
if there exist two distinct indices $1 \le k_1 , k_2 \le r$ with
$i_{k_1} < i_{k_2}$ and $j_{k_1} < j_{k_2}$.

We say that the set contains an \emph{$\omega_1$-2-chain} $(j_{k_1},i_{k_1}) \prec (j_{k_2},i_{k_2})$
if there exist two distinct indices $1 \le k_1 , k_2 \le r$ such that
$i_{k_1} < i_{k_2}$, $j_{k_1} < j_{k_2}$ and $j_{k_2} > n+1$.
Similarly, it contains an \emph{$\omega_3$-2-chain} $(j_{k_1},i_{k_1}) \prec (j_{k_2},i_{k_2})$
if there exist two distinct indices $1 \le k_1 , k_2 \le r$ with
$i_{k_1} < i_{k_2}$, $j_{k_1} < j_{k_2}$ and $i_{k_1} < 0$.

We say that the set contains an \emph{$\omega_1$-3-chain}
$(j_{k_1},i_{k_1}) \prec (j_{k_2},i_{k_2}) \prec (n+n_1+2-i_{k_3},\,n+n_1+2-j_{k_3})$
if there exist three distinct indices $1 \le k_1, k_2 , k_3 \le r$ such that
\begin{align*}
0 \le i_{k_3} \le i_{k_1} < i_{k_2}, \quad
j_{k_1} < j_{k_2} \le n+1 < j_{k_3} < n + n_1 + 2 - i_{k_2}.
\end{align*}
Similarly, it contains an \emph{$\omega_3$-3-chain}
$(n_2-i_{k_1},\,n_2-j_{k_1}) \prec (j_{k_2},i_{k_2}) \prec (j_{k_3},i_{k_3})$
if there exist three distinct indices $1 \le k_1, k_2, k_3 \le r$ such that
\begin{align*}
n_2 - j_{k_2} < i_{k_1} < 0 \le i_{k_2} < i_{k_3}, \quad
j_{k_2} < j_{k_3} \le j_{k_1} \le n.
\end{align*}

Finally, we say that the set $\{(j_1,i_1),(j_2,i_2),\dots,(j_r,i_r)\}$ is
\emph{free of  $\,\Omega$-chains} if it contains none of the following configurations:
a 3-chain, an $\omega_1$-2-chain, an $\omega_1$-3-chain, an $\omega_3$-2-chain, or an $\omega_3$-3-chain.

All index collections below are multisets, with repetitions counted. For $m\geq0$, we set
\begin{equation*}\wht J_1^m=\{\{i_1,i_2,...,i_m\}\mid i_s\in J_1\},\quad\wht J_3^m=\{\{j_1,j_2,...,j_m\}\mid j_t\in J_3\}.\end{equation*}
Let \begin{equation*}
f=\left(\prod_{t_1=1}^{k_1} x_{s_{1,t_1}}\right)  \left(\prod_{t_2=1}^{k_2} y_{s_{2,t_2}}\right)  \left( \prod_{t_3=1}^{k_3} z_{j_{t_3},i_{t_3}}\right)
\left( \prod_{t_4=1}^{k_4} \omega_{p_{t_4},q_{t_4}}\right)
\left( \prod_{t_5=1}^{k_5} \omega_{p'_{t_5},q'_{t_5}}\right) g\in
\msr F
\end{equation*}
where $g\in\mbb F$; $s_{1,t_1},i_{t_3},p_{t_4},q_{t_4} \in J_1$; $s_{2,t_2},j_{t_3},p'_{t_5},q'_{t_5}  \in J_3$ and $k_1$, $k_2$, $k_3$, $k_4$, $k_5 \in \mbb N$.
We define
\begin{eqnarray*}
& &\msr I_1\left( f \right) = \{s_{1,1},...,s_{1,k_1},i_1,...,i_{k_3}
,p_1,...,p_{k_4},q_1,...,q_{k_4}\} \in \wht J_1^{k_1+k_3+2k_4},
 \\
& &
\msr I_3\left( f \right) = \{s_{2,1},...,s_{2,k_2},j_1,...,j_{k_3},p'_1,...,p'_{k_5},q'_1,...,q'_{k_5}\}
\in \wht J_3^{k_2+k_3+2k_5}.\end{eqnarray*}

Denote by $Z^{(m)}$ the set of (monic) monomials of degree $m$ in the variables
$\{z_{j,i}\mid j\in J_3,\; i\in J_1\}$.
Fix $\mathcal{I}_1\in \widehat{J}_1^{\,k_1+k_3+2k_4}$ and
$\mathcal{I}_3\in \widehat{J}_3^{\,k_2+k_3+2k_5}$. We define
\begin{align}\label{4.12}
\begin{split}
&  G_{k_1,k_2,k_3,k_4,k_5}^{{\cal I}_1,{\cal I}_3}
 \\
= \,&
 \Big\{ f\,=\,g_1g_3h\left(\prod_{t_1=1}^{k_1} x_{s_{1,t_1}}\right)  \left(\prod_{t_2=1}^{k_2} y_{s_{2,t_2}}\right)
\Big\vert\;  h\in   Z^{(k_3)} ,\;g_1\in \Omega_1^{(k_4)},g_3\in \Omega_3^{(k_5)},\;
 \\&
\msr I_1(f)={\cal I}_1, \;
\msr I_3(f)={\cal I}_3,\;\msr I(f) \text{ is free of $\Omega$-chains; }
 \\ &
1\leq s_{1,1}\leq\cdots\leq s_{1,k_1} \leq n_1,\; n_2+1\leq  s_{2,1}\leq\cdots \leq s_{2,k_2} \leq n.
\Big\}
\end{split}\end{align}
and for $k \in \mathbb{N}$,
\begin{align}\label{4.13}\begin{split}
S^{{\cal I}_1,{\cal I}_3}_k = \bigcup_{\sum_{i=1}^5k_i=k} G_{k_1,k_2,k_3,k_4,k_5}^{{\cal I}_1,{\cal I}_3}
\end{split}\end{align}
for $k \in \mbb N$.
Set
\begin{align*} \begin{split}
 V_{k_1,k_2,k_3,k_4,k_5}^{{\cal I}_1,{\cal I}_3} = \operatorname{Span} \{  G_{k_1,k_2,k_3,k_4,k_5}^{{\cal I}_1,{\cal I}_3}  \},\quad
V^{{\cal I}_1,{\cal I}_3}_k=\operatorname{Span}\{S^{{\cal I}_1,{\cal I}_3}_k\}.
\end{split}\end{align*}

According to Lemma~4.9 in \cite{ZX2}, we obtain the following statement.
\begin{lemma}\label{lem:4.1}
For $k_1,k_2,k_3,k_4  \in \mathbb{N}$, ${\cal I}_1\in  \wht J_1^{k_1+k_3+2k_4}$ and ${\cal I}_3\in \wht J_3^{k_2+k_3}$, the set
$  \phi   \left(G_{k_1,k_2,k_3,k_4,0 }^{{\cal I}_1,{\cal I}_3} \right)$ is linearly independent in $\msr A$.
\end{lemma}
Symmetrically, for $k_1,k_2,k_3, k_5 \in \mathbb{N}$, ${\cal I}_1\in  \wht J_1^{k_1+k_3 }$ and ${\cal I}_3\in \wht J_3^{k_2+k_3+2k_5}$, the set
$  \phi   \left(G_{k_1,k_2,k_3,0,k_5 }^{{\cal I}_1,{\cal I}_3} \right)$ is linearly independent  in $\msr A$.

Consequently, we have the following result.
\begin{lemma}\label{lem:4.2}
For $k_1,k_2,k_3,k_4,k_5  \in \mathbb{N}$, ${\cal I}_1\in  \wht J_1^{k_1+k_3+2k_4}$ and ${\cal I}_3\in \wht J_3^{k_2+k_3+2k_5}$, the set
$   \phi   \left(G_{k_1,k_2,k_3,k_4,k_5 }^{{\cal I}_1,{\cal I}_3} \right)$ is linearly independent  in $\msr A$.
\end{lemma}

\begin{proof}
Note that if $G_{k_1,k_2,k_3,k_4,k_5}^{{\cal I}_1,{\cal I}_3} \neq \emptyset$, then $(k_2+k_5)k_4=0$ in (\ref{4.12}).
Indeed, if both $k_2+k_5>0$ and $k_4>0$, there would exist $\omega_{p,q} \mid g_1$ for some $(p,q) \in O_1$, and then $\mathcal{I}(f)$ would contain an $\omega_1$-2-chain $(s,0) \prec (n+n_1+2-p,q)$ for some $s\in J_3$, or $(q',n_2-p')\prec(n+n_1+2-p,q)$ for some $(p',q')\in O_3$.
Similarly, we also have $(k_1+k_4)k_5=0$ whenever $G_{k_1,k_2,k_3,k_4,k_5}^{{\cal I}_1,{\cal I}_3} \neq \emptyset$.
Hence $G_{k_1,k_2,k_3,k_4,k_5}^{{\cal I}_1,{\cal I}_3} = \emptyset$ whenever $k_4 \neq 0$ and $k_5 \neq 0$.
If $k_4=0$ or $k_5=0$, the required linear independence follows directly from Lemma~\ref{lem:4.1} (or its symmetric counterpart).

\end{proof}

\begin{lemma}\label{lem:4.3}
For $k_1,k_2,k_3,k_4,k_5  \in \mathbb{N}$, ${\cal I}_1\in  \wht J_1^{k_1+k_3+2k_4}$ and ${\cal I}_3\in \wht J_3^{k_2+k_3+2k_5}$, the set $   \phi   \left(S_k^{{\cal I}_1,{\cal I}_3} \right)$ is linearly independent in $\msr A$, where $k=\sum_{i=1}^5k_i$.
\end{lemma}

\begin{proof}
From \eqref{4.12} and \eqref{4.13} we have
\begin{align*}
S^{{\cal I}_1,{\cal I}_3}_k =  \left(\bigcup_{\sum_{i=1}^4k_i=k } G_{k_1,k_2,k_3,k_4,0}^{{\cal I}_1,{\cal I}_3} \right) \bigcup
\left(\bigcup_{\sum_{1\leq i\leq 5,i\neq 4} k_i=k } G_{k_1,k_2,k_3,0,k_5}^{{\cal I}_1,{\cal I}_3}\right). \end{align*}
Assume that $f \in \operatorname{Span}\{S_k^{\mathcal{I}_1,\mathcal{I}_3}\}$ satisfies $\phi(f)=0$.
Then $f$ can be written as
\begin{align} \label{a4.17}
f=\sum_{\sum_{i=1}^5 k_i=k }  f^{k_1,k_2,k_3,k_4,k_5} \end{align}
where $ f^{k_1,k_2,k_3,k_4,k_5} \in \operatorname{Span} \{G_{k_1,k_2,k_3,k_4,k_5}^{{\cal I}_1,{\cal I}_3}\} $ when $G_{k_1,k_2,k_3,k_4,k_5}^{{\cal I}_1,{\cal I}_3}\neq \emptyset$ and
$f^{k_1,k_2,k_3,k_4,k_5}=0$, otherwise.

Define
\begin{align*} \begin{split}
&f_{1,r}=\sum_{k_1+k_2=k-r } f^{k_1,k_2,r,0,0},\;  \\
&f_{2,r} =\sum_{  k_3+k_4= r,k_4>0 } f^{k-r,0,k_3,k_4,0},\; \\
&f_{3,r}=\sum_{  k_3+k_5= r,k_5>0 } f^{0,k-r,k_3,0,k_5},\; . \end{split}\end{align*}
for $r=0,1,\dots,k$. Then \begin{align*}
f=\sum_{r=0}^k f_{1,r}+\sum_{r=1}^k f_{2,r}+\sum_{r=1}^k f_{3,r}.\end{align*}

Since $\deg (\phi(f_{i,r}))=r+k$, we obtain $\phi\bigl(\sum_{i=1}^{3}f_{i,r}\bigr)=0$ for each $r \in \overline{0,k}$.

Introduce the weighted degree
\[
\mfk q_1(x^\alpha y^\beta)=\sum_{i\in J_1}(\alpha_i+\beta_i).
\]
For fixed $r$, the nonzero images of the summands in $f_{1,r}$, $f_{2,r}$ and $f_{3,r}$ have respective $\mfk q_1$-degrees
\[
k_1+r\in[r,k],\qquad k+r-k_3\in[k+1,k+r],\qquad k_3\in[0,r-1].
\]
These three intervals are disjoint. Within each group the displayed degree determines the remaining index. Comparing homogeneous components in $\phi(f_{1,r}+f_{2,r}+f_{3,r})=0$ therefore gives
$\phi(f^{k_1,k_2,k_3,k_4,k_5})=0$ for every summand.

Summing up, we have shown $\phi(f^{k_1,k_2,k_3,k_4,k_5}) = 0$ for every summand in \eqref{a4.17}.
By Lemma~\ref{lem:4.2}, each $f^{k_1,k_2,k_3,k_4,k_5}$ must be zero, and hence $f=0$.
Thus $\phi(S_k^{\mathcal{I}_1,\mathcal{I}_3})$ is linearly independent.

\end{proof}

For $i\in J_1$; $j\in J_3$; $i',j' \in J_s$ with $j'>i'$, $s=1,3$,
we set
\begin{align*}
\omega_{i,0}=x_i,\;\; \omega_{n+1,j}=-y_j,\;\;\omega_{i,i}=\omega_{n+1,n+1}=\omega_{j,j}=\omega_{0,0}=0, \text{ for }i\in J_1,\;  j\in J_3;
\end{align*}
and\begin{align*}
\omega_{i',j'}=-\omega_{j',i'}, \text{  for } j',i'\in J_1' \text{ or }J_3' \text{ with } j'>i' .
\end{align*}

When $i_3',i_3,j_3,j_3'\in J_3'$ and $i_1',i_1,j_1,j_1' \in J_1'$, we denote
\begin{align*}
\Lambda( z_{i_3,i_1}\omega_{j'_1,j_1} )
=  z_{i_3,i_1}\omega_{j'_1,j_1} - z_{i_3,j_1}\omega_{j'_1,i_1}+ z_{i_3,j'_1}\omega_{j_1,i_1},
\end{align*}
\begin{align*}
\Lambda( z_{i_3,i_1}\omega_{j'_3,j_3} )
=  z_{i_3,i_1}\omega_{j'_3,j_3} - z_{j'_3,i_1}\omega_{i_3,j_3}+ z_{j_3,i_1}\omega_{i_3,j'_3},
\end{align*}
\begin{align*}
\Lambda( \omega_{j'_1,j_1}\omega_{j'_3,j_3} )
=  \omega_{j_3',j_3} \omega_{j'_1,j_1} +z_{j'_3,j'_1}z_{j_3,j_1}- z_{j'_3,j_1}z_{j_3,j'_1},
\end{align*}
\begin{align*}
\Lambda( \omega_{j'_1,j_1}\omega_{i'_1,i_1} )
= \omega_{j'_1,j_1}\omega_{i'_1,i_1}- \omega_{j'_1,i_1}\omega_{i'_1,j_1}+ \omega_{j'_1,i'_1}\omega_{i_1,j_1},
\end{align*}
\begin{align*}
\Lambda( \omega_{j'_3,j_3}\omega_{i'_3,i_3} )
=   \omega_{j'_3,j_3}\omega_{i'_3,i_3}  + \omega_{j'_3,i'_3}\omega_{i_3,j_3} - \omega_{j'_3,i_3}\omega_{i'_3,j_3}.
\end{align*}

Let
\begin{align}  \begin{split}\label{b4.29}
\mfk R =\Big\{ &
\Lambda( z_{i_3,i_1}\omega_{j'_1,j_1} ),\Lambda( z_{i_3,i_1}\omega_{j'_3,j_3} ),
\Lambda( \omega_{j'_1,j_1}\omega_{j'_3,j_3} ) ,
\Lambda( \omega_{j'_1,j_1}\omega_{i'_1,i_1} ),
\Lambda( \omega_{j'_3,j_3}\omega_{i'_3,i_3} )    \mid   \\&
 \text{ } i_3',i_3,j_3,j_3'\in J_3'
\text{  and } i_1',i_1,j_1,j_1' \in J_1' ,
\Big\} .\end{split}   \end{align}
and let $\msr R=\la \mfk R \ra$ be an ideal of $\msr F$.

\begin{lemma}
\label{lem:4.4}
For $k\in\mbb N$, we have $\ker(\phi|_{\msr F_k})=\msr R\cap\msr F_k$. Consequently, $\ker\phi=\msr R$.\end{lemma}

\begin{proof}

For distinct $j_1,j_2,j_3 \in J_3'$ and distinct $i_1,i_2,i_3  \in J_1'$, we denote
\begin{align*}  \begin{split}
\Lambda(z_{j_1,i_1}z_{j_2,i_2}z_{j_3,i_3})=
&z_{j_1,i_1}z_{j_2,i_2}z_{j_3,i_3}+z_{j_1,i_2}z_{j_2,i_3}z_{j_3,i_1}+z_{j_1,i_3}z_{j_2,i_1}z_{j_3,i_2} \\
&-z_{j_1,i_1}z_{j_2,i_3}z_{j_3,i_2}-z_{j_1,i_2}z_{j_2,i_1}z_{j_3,i_3}
 -z_{j_1,i_3}z_{j_2,i_2}z_{j_3,i_1}.
\end{split}   \end{align*}

For $n_2+1\leq s_1<s_2\leq n+1$ and $1\leq t_0\leq t_1<t_2<t_3\leq n_1$, we denote
\begin{align*}  \begin{split}
\Lambda(z_{s_1,t_1}z_{s_2,t_2}\omega_{t_3,t_0})=
&z_{s_1,t_1}z_{s_2,t_2}\omega_{t_3,t_0}+z_{s_1,t_2}z_{s_2,t_3}\omega_{t_1,t_0}+z_{s_1,t_3}z_{s_2,t_1}\omega_{t_2,t_0} \\
&-z_{s_1,t_3}z_{s_2,t_2}\omega_{t_1,t_0}-z_{s_1,t_1}z_{s_2,t_3}\omega_{t_2,t_0}
 -z_{s_1,t_2}z_{s_2,t_1}\omega_{t_3,t_0}.
\end{split}   \end{align*}
For $0\leq t'_1< t'_2 \leq n_1$ and $n_2+1 \leq s'_0 < s'_1<s'_2 \leq s'_3 \leq n$, we denote
\begin{align*}  \begin{split}
\Lambda(z_{s'_1,t'_1}z_{s'_2,t'_2}\omega_{s'_3,s'_0})=
&z_{s'_1,t'_1}z_{s'_2,t'_2}\omega_{s'_3,s'_0}+z_{s'_0,t'_1}z_{s'_1,t'_2}\omega_{s'_3,s'_2}+ z_{s'_2,t'_1}z_{s'_0,t'_2}\omega_{s'_3,s'_1} \\
&-z_{s'_2,t'_1}z_{s'_1,t'_2}\omega_{s'_3,s'_0}
-z_{s'_0,t'_1}z_{s'_2,t'_2}\omega_{s'_3,s'_1}
 -z_{s'_1,t'_1}z_{s'_0,t'_2}\omega_{s'_3,s'_2}.
\end{split}   \end{align*}

Let
\begin{align*}  \begin{split}
\mfk R_0 =\mfk R  \cup \Big\{ &
\Lambda(z_{j_1,i_1}z_{j_2,i_2}z_{j_3,i_3}),\Lambda(z_{s_1,t_1}z_{s_2,t_2}\omega_{t_3,t_0}),
\Lambda(z_{s'_1,t'_1}z_{s'_2,t'_2}\omega_{s'_3,s'_0})
\mid\\
&\text{ distinct }  j_1,j_2,j_3  \in J_3',
\text{  and distinct } i_1,i_2,i_3   \in J_1'\\
& n_2+1\leq s_1<s_2\leq n+1 ,1\leq t_0\leq t_1<t_2<t_3\leq n_1,  \\
&0\leq t'_1< t'_2 \leq n_1  ,n_2+1 \leq s'_0 < s'_1<s'_2 \leq s'_3 \leq n
\Big\} ,\end{split}   \end{align*}
and $\msr R_0=\la \mfk R_0 \ra$.

Since
\begin{align*}  \begin{split}
\Lambda(z_{j_1,i_1}z_{j_2,i_2}z_{j_3,i_3})=
&-\omega_{j_2,j_1}\Lambda(z_{j_3,i_1}\omega_{i_3,i_2})
+z_{j_3,i_3}\Lambda(\omega_{i_2,i_1}\omega_{j_2,j_1}) \\
&-z_{j_3,i_2}\Lambda(\omega_{i_3,i_1}\omega_{j_2,j_1})
+z_{j_3,i_1}\Lambda(\omega_{i_3,i_2}\omega_{j_2,j_1}) ,
\end{split}   \end{align*}
\begin{align*}  \begin{split}
\Lambda(z_{s_1,t_1}z_{s_2,t_2}\omega_{t_3,t_0}) =
&z_{s_1,t_2}\Lambda(z_{s_2,t_0}\omega_{t_3,t_1})
-z_{s_1,t_1}\Lambda(z_{s_2,t_0}\omega_{t_3,t_2})\\
&-z_{s_1,t_3}\Lambda(z_{s_2,t_0}\omega_{t_2,t_1})
+z_{s_2,t_0}\Lambda( z_{s_1,t_1}\omega_{t_3,t_2}) ,
\end{split}   \end{align*}
and
\begin{align*}  \begin{split}
\Lambda(z_{s'_1,t'_1}z_{s'_2,t'_2}\omega_{s'_3,s'_0})=
&z_{s_0',t_1'}\Lambda(z_{s'_3,t'_2}\omega_{s'_2,s_1'})
+z_{s_2',t_1'}\Lambda(z_{s'_3,t'_2}\omega_{s'_1,s_0'})\\
&-z_{s_1',t_1'}\Lambda(z_{s'_3,t'_2}\omega_{s'_2,s_0'})
-z_{s_3',t_2'}\Lambda(z_{s'_2,t'_1}\omega_{s'_1,s_0'}),
\end{split}   \end{align*}
we have $\msr R_0 =\msr R$.

It can be verified that $\phi(\msr R_0)=\{0\}$ and then $\msr R_0\subseteq\ker\phi$.
We only need to show that $\ker(\phi|_{\msr F_k})\subseteq\msr R_0$.

Suppose that $f=\sum_i f_i\in\ker(\phi|_{\msr F_k})$, where the $f_i$ are scalar multiples of monomials.
Since $\phi$ preserves the index multisets $\msr I_1$ and $\msr I_3$, we may consider each pair of multisets separately.
Thus we may assume that there exist ${\cal I}_1\in\widehat J_1^{m_1}$ and ${\cal I}_3\in\widehat J_3^{m_2}$ such that
$\msr I_1(f_i)={\cal I}_1$ and $\msr I_3(f_i)={\cal I}_3$ for each $i$.
All the reductions below preserve these multisets and the degree $k$ in $\msr F$.

First we remove products containing both types of $\omega$ variables. The quadratic relations in $\mfk R$ give
\begin{align*}
\omega_{j_3',j_3}\omega_{j_1',j_1}
&\equiv-z_{j_3',j_1'}z_{j_3,j_1}+z_{j_3',j_1}z_{j_3,j_1'},\\
y_s\omega_{p,q}&\equiv z_{s,q}x_p-z_{s,p}x_q,\\
x_t\omega_{p',q'}&\equiv z_{q',t}y_{p'}-z_{p',t}y_{q'}
\pmod{\msr R_0},
\end{align*}
where $p,q,t\in J_1$, $s,p',q'\in J_3$ and the indices in each $\omega$ factor are distinct.
Each replacement decreases the total number of $\omega$ factors.
Consequently, after finitely many replacements, every monomial contains either only $z$ variables, only $\Omega_1$ and $z$ variables with no $y_s$ factor, or only $\Omega_3$ and $z$ variables with no $x_t$ factor.

We next reduce the first two types, using the same index order as in the chain argument.
For a monomial $h$, arrange the pairs in $\msr I(h)$ in increasing lexicographic order, with repetitions retained, and compare the resulting lists lexicographically.
Multiplication by a fixed monomial preserves this comparison. There are only finitely many such lists for monomials of degree $k$.
The following replacements strictly increase the list whenever the term on the left occurs:
\begin{align*}
z_{s,i}\omega_{p,q}
&\equiv z_{s,q}\omega_{p,i}-z_{s,p}\omega_{q,i}
&& (1\leq i<q<p\leq n_1),\\
\omega_{p,q}\omega_{r,t}
&\equiv\omega_{p,t}\omega_{r,q}-\omega_{p,r}\omega_{t,q}
&& (1\leq q<t<r<p\leq n_1).
\end{align*}
Here the first congruence holds for $s\in J_3'$, and both congruences are modulo $\msr R_0$.
The first term in each congruence contains an $\omega_1$-2-chain.
Conversely, every $\omega_1$-2-chain in these two types of monomials has one of these forms.

For $j_1<j_2<j_3$ and $i_1<i_2<i_3$, the first term
$z_{j_1,i_1}z_{j_2,i_2}z_{j_3,i_3}$ in
$\Lambda(z_{j_1,i_1}z_{j_2,i_2}z_{j_3,i_3})$
has a strictly smaller list than each of the other nonzero terms.
The same is true of the first term
$z_{s_1,t_1}z_{s_2,t_2}\omega_{t_3,t_0}$ in
$\Lambda(z_{s_1,t_1}z_{s_2,t_2}\omega_{t_3,t_0})$,
where $1\leq t_0\leq t_1<t_2<t_3\leq n_1$.
These two cubic relations therefore remove, respectively, 3-chains consisting of $z$ factors and $\omega_1$-3-chains.
A 3-chain involving an $\Omega_1$ factor already contains an $\omega_1$-2-chain.
Solving each relation for its first term and continuing the reductions must terminate, since the lists strictly increase in a finite set.
These replacements do not introduce an $\Omega_3$ factor or a $y_s$ factor into a monomial of the second type.
We obtain only monomials free of $\Omega$-chains.

The third type is treated in the same way, with the two sets of outer indices interchanged.
More explicitly, apply the reflection $u\mapsto n+1-u$ to the indices and interchange $x$ and $y$.
It sends $z_{j,i}$ to $z_{n+1-i,n+1-j}$, interchanges the two types of $\omega$ variables up to sign, and replaces $(n_1,n_2)$ by $(n-n_2,n-n_1)$.
The $\omega_3$-2-chains and $\omega_3$-3-chains become the corresponding $\omega_1$-chains.
Thus the reflected lexicographic list gives the same terminating reduction, using
$\Lambda(z_{i_3,i_1}\omega_{j_3',j_3})$,
$\Lambda(\omega_{j_3',j_3}\omega_{i_3',i_3})$ and the third cubic family above.
This also produces only monomials free of $\Omega$-chains.

We have therefore obtained a polynomial $f'$ such that
\[
f\equiv f'\pmod{\msr R_0},\qquad
f'\in V_k^{{\cal I}_1,{\cal I}_3}.
\]
Since $\phi(\msr R_0)=\{0\}$, we have $0=\phi(f)=\phi(f')$.
By Lemma~\ref{lem:4.3}, $f'=0$; that is, $f\in\msr R_0=\msr R$.
Hence $\ker(\phi|_{\msr F_k})\subseteq\msr R\cap\msr F_k$, and the reverse inclusion has already been proved.
Finally, assign weight one to $x_i$ for $i\in J_1$ and $y_j$ for $j\in J_3$, and weight zero to the remaining variables of $\msr A$.
Every generator of $\msr F$ has image of weight one.
Thus $\phi$ is graded for this weight and the usual grading of $\msr F$, and the assertions for all $k$ give $\ker\phi=\msr R$.
\end{proof}
\pse

\section{Associated Varieties and the Pfaffian Variety}

We now compute the annihilator of the associated graded module and identify its zero set. Throughout this section, $M=\msr H_{\la\ell\ra}$, $\mfk g=\mfk o(2n)$ and $\ell\leq0$. We first assume $n_1<n_2$ and use $M_0=V_0$ from Section 3; the case $n_1=n_2$ is treated in the proof of Theorem~\ref{the:5.1}.

For $p\geq1$, put
\begin{equation*}
\mfk k_{(p)}=\{\xi\in U_p(\mfk g):\xi(M_k)\subseteq M_{k+p-1}\text{ for every }k\geq0\}.
\end{equation*}
The associated graded space $\overline M$ in \eqref{1.11} is an $S(\mfk g)$-module. Write $S_{(p)}(\mfk g)=S^p(\mfk g)=U_p(\mfk g)/U_{p-1}(\mfk g)$ and
\begin{equation*}
I_{(p)}=\operatorname{Ann}_{S(\mfk g)}(\overline M)\cap S_{(p)}(\mfk g).
\end{equation*}
Then $\operatorname{Ann}_{S(\mfk g)}(\overline M)=\bigoplus_{p\geq1}I_{(p)}$.

For convenience, we introduce the notation
\begin{align}\label{1.15}
\td{E}_{i,j} =\left\{\begin{array}{ll}
E_{i,j}-E_{n+j,n+i}  &\mbox{if}\;
i,j \in \ol{1,n } ;\\
E_{i,j}-E_{j-n,n+i}  &\mbox{if}\;  i \leq n<j ;\\
E_{i,j}-E_{n+j,i-n}  &\mbox{if}\;   j\leq n<i ;\\
E_{i,j}-E_{j-n,i-n}  &\mbox{if}\; i,j>n,
\end{array}\right.
\end{align}
for $i,j\in \ol{1,2n}$.
For the coordinate indices, put
\[
A=J_1\cup(n+J_3),\qquad B=J_2\cup(n+J_2),\qquad C=J_3\cup(n+J_1).
\]
Thus $A,B,C$ partition $\ol{1,2n}$, and $\theta$ exchanges $A$ and $C$ and preserves $B$.
Define the sets
\begin{align*}
L_0=(J_3\times J_1 ) \cup (J_3\times (J_3+n) ) \cup ((J_1+n)\times J_1 ) \cup ((n+J_1)\times (n+J_3)) ,
\end{align*} 
\begin{align}\begin{split}\label{1.17}
L =&L_0 \cup (J_3 \times J_2 ) \cup (J_2\times J_1 ) \cup ( J_2 \times (J_3+n) ) \\
&\cup (J_3\times(J_2+n))\cup((J_2+n)\times J_1)\cup((J_1+n)\times J_2)\\ &\cup((J_2+n)\times(J_3+n))\cup((J_1+n)\times(J_2+n)),
\end{split}\end{align}
 and $L^c=(\{1,2, \dots, 2n  \} \times \{  1,2, \dots, 2n  \} ) \setminus L$.

\begin{lemma}\label{lem:5.1}
For $p\geq1$, we have
\begin{align}\label{5.6}
 I_{(p)}   =  \left\lbrace \eta+U_{p-1}( \mathfrak{g} ) \in S_{(p)} \left( \mathfrak{g} \right)  | \eta( M_0) \subseteq M_{p-1}  \right\rbrace. \end{align}
\end{lemma}
\begin{proof}
Suppose that $ \eta+U_{p-1}( \mathfrak{g} ) \in S_{(p)} \left( \mathfrak{g} \right)$ and $  \eta( M_0) \subseteq M_{p-1}$.
Then
\begin{align*}\begin{split}
 \eta(M_k) +M_{k+p-1}&\subseteq   \eta \left(  U_k \left( \mathfrak{g} \right) (M_0)\right) +M_{k+p-1}\\
&\subseteq    U_k \left( \mathfrak{g} \right)\eta( M_0)  +M_{k+p-1}
\subseteq M_{k+p-1}.
\end{split}\end{align*}
Here $[U_p(\mfk g),U_k(\mfk g)]\subseteq U_{p+k-1}(\mfk g)$. Hence the symbol of $\eta$ belongs to $I_{(p)}$. Conversely, if the symbol of $\eta\in U_p(\mfk g)$ annihilates $\overline M$, its action on $\overline M_0=M_0$ gives $\eta(M_0)\subseteq M_{p-1}$.
Therefore, (\ref{5.6}) holds.

\end{proof}

For $(j,i)\notin L$, the formulas \eqref{3.1}--\eqref{3.5} show that
\begin{equation*}
\widetilde E_{j,i}(M_0)\subseteq M_0.
\end{equation*}
Consequently these elements belong to $I_{(1)}$ by Lemma~\ref{lem:5.1}. Conversely, the quadratic multiplication parts of the remaining independent generators are linearly independent polynomials. Applying a nonzero linear combination to a nonzero element of $M_0$, which is homogeneous of ordinary degree $-\ell$, gives a nonzero term of degree $2-\ell$. Hence
\begin{equation}\label{5.11}
I_{(1)}=\operatorname{Span}\{\widetilde E_{j,i}:(j,i)\notin L\}.
\end{equation}
The last two rectangles in \eqref{1.17} include the partners required by $\widetilde E_{i,j}=-\widetilde E_{\theta(j),\theta(i)}$, where $\theta(i)=n+i$ for $i\leq n$ and $\theta(i)=i-n$ for $i>n$.

Set
\begin{align}\begin{split} \label{a1.18}
\Delta_{i_1,i_2}^{j_1,j_2}=\td{E}_{j_1,i_1} \td{E}_{j_2,i_2}-\td{E}_{j_1,i_2} \td{E}_{j_2,i_1}\end{split}\end{align}
for $i_1,i_2, j_1,j_2 \in\ol{1,2n}$,
and
\begin{align*}
I_{\Delta_2}= \{ \Delta_{i_1,i_2}^{j_1,j_2} \mid (j_1,i_1),(j_2,i_2),(j_2,i_1),(j_1,i_2)\in L\setminus L_0 \} .\end{align*}

For $ j_1, i_1 \in J_1$ and $(s,t)\in L$, set
\begin{align}  \label{a1.19}
\ol{\Delta}^{n+j_1,s}_{i_1,t} =
\left\{\begin{array}{ll}
 \td{E}_{s,t}\td{E}_{n+j_1,i_1} -\td{E}_{s, j_1}\td{E}_{t-n,i_1} +\td{E}_{s,i_1}\td{E}_{t-n,j_1}   &\mbox{if}\;
t>n ;\\
  \td{E}_{s,t}\td{E}_{n+j_1,i_1} -\td{E}_{s,i_1}\td{E}_{n+j_1,t}+\td{E}_{s,j_1}\td{E}_{n+i_1,t} &\mbox{if}\;  t\leq n .
\end{array}\right.
\end{align}

For $ j_3, i_3 \in J_3$ and $(s,t)\in L$, define
\begin{align}    \label{a1.20}
\ol{\Delta}^{ j_3,s}_{n+i_3,t} =
\left\{\begin{array}{ll}  \td{E}_{ j_3,n+i_3}\td{E}_{s,t} -\td{E}_{ j_3,s-n}\td{E}_{ i_3,t}+  \td{E}_{i_3,s-n } \td{E}_{ j_3,t} &\mbox{if}\;
s>n ;\\
 \td{E}_{ j_3,n+i_3}\td{E}_{s,t} - \td{E}_{s,n+i_3}\td{E}_{ j_3,t}+\td{E}_{s,n+j_3}\td{E}_{ i_3,t} &\mbox{if}\;  s\leq n .
\end{array}\right.
\end{align}

The products in \eqref{a1.18}--\eqref{a1.20} are understood as symbols in $S(\mfk g)$. An ordered product in $U(\mfk g)$ is used when it is applied to a vector.

For $v_0\in M_0$, we have
\[
\partial_{x_i}(v_0)=\partial_{y_j}(v_0)=0
\qquad(i\in J_2\cup J_3,\ j\in J_1\cup J_2).
\]
Put $F_{ab}=\widetilde E_{\theta(a),b}=-F_{ba}$ and $g_{ab}=\delta_{a,\theta(b)}$. For $a,b\in A$ and $c,d\in B$, direct calculation gives
\begin{align}\label{eq:block-minors}
F_{ac}F_{bd}-F_{ad}F_{bc}&=F_{ab}F_{cd},\nonumber\\
F_{ca}F_{db}-F_{cb}F_{da}&=F_{ab}(F_{cd}-g_{cd}).
\end{align}
Since $F_{cd}(v_0)=0$, the first of these minors kills $M_0$, while the second sends $M_0$ into $M_1$. The two orientations give the same minors, up to sign, in the symmetric algebra. Consequently, Lemma~\ref{lem:5.1} gives
\begin{equation}\label{5.15}
I_{\Delta_2}\subseteq I_{(2)}.
\end{equation}

We also record the ordered identity underlying the three-term relations:
\begin{equation}\label{eq:ordered-pf}
F_{ab}F_{cd}-F_{ac}F_{bd}+F_{ad}F_{bc}
=g_{ab}F_{cd}-g_{ac}F_{bd}+g_{ad}F_{bc}.
\end{equation}
Before the Fourier transformation, substitute
$F_{ab}=z_{\theta(a)}\partial_{z_b}-z_{\theta(b)}\partial_{z_a}$.
The terms with two derivatives cancel, and the terms with one derivative give the right-hand side. The Fourier transformation preserves this identity. The first identity in \eqref{eq:block-minors} follows because all pairings between $A$ and $A\cup B$ vanish; the second follows by one commutator. This verification also applies in Section 6 with $\theta(0)=0$.

In the displayed orderings of \eqref{a1.19}--\eqref{a1.20}, the contraction terms in \eqref{eq:ordered-pf}, and the commutators needed to arrange the factors, vanish for $(s,t)\in L$. Thus these formulas give
\[
\overline\Delta^{n+j_1,s}_{i_1,t}
=\overline\Delta^{j_3,s}_{n+i_3,t}=0
\]
as operators. Their symbols therefore belong to $I_{(2)}$.

\begin{lemma}\label{lem:5.2}
Let $J$ be the homogeneous ideal generated by $I_{(1)}$, $I_{\Delta_2}$ and the quadrics $\overline\Delta^{n+j_1,s}_{i_1,t}$, $\overline\Delta^{j_3,s}_{n+i_3,t}$, where $i_1,j_1\in J_1$, $i_3,j_3\in J_3$ and $(s,t)\in L$. Then
\begin{equation}\label{5.20}
I_{(p)}=J\cap S_{(p)}(\mfk g)\qquad(p\geq1).
\end{equation}
\end{lemma}

\begin{proof}

The inclusion $J\subseteq\operatorname{Ann}_{S(\mfk g)}(\overline M)$ follows from \eqref{5.11}, \eqref{5.15} and the three-term operator identities above. We prove the other inclusion by the highest degree terms of the oscillator operators.

For comparison with Section 2, the auxiliary weight is
\[
\mfk d_2(x^\alpha y^\beta)=-\sum_{i\in J_3}\alpha_i-\sum_{i\in J_1}\beta_i.
\]
The elements of $W$ have $\mfk d_1$-weight $2$ and $\mfk d_2$-weight $-1$, so Lemma~\ref{lem:2.2} applies whenever $W$ is nonempty. On a monomial of signed degree $\ell$, ordinary degree equals $2(\mfk d_1+\mfk d_2)-\ell$. We use ordinary degree explicitly in the following comparison.

Modulo $\langle I_{(1)}\rangle$, choose a representative of a homogeneous annihilator symbol as follows.
Let $\xi \in \mfk k_{(p)}$ and write
\begin{align}\label{a5.19}
\xi =  \sum_{l=1}^m c_l\left(
\prod_{r=1}^7 \left(\prod_{s=1}^{k_{r,l}} \td{E}_{j_{r,l,s},i_{r,l,s}}\right)
\right) ;
\end{align}
where $m,k_{1,l},\dots,k_{7,l}\in\mbb N$ satisfy $\sum_{r=1}^7 k_{r,l}=p$, and
 \begin{align}\begin{split}\label{a5.20}
&(j_{1,l,s},i_{1,l,s})\in J_2\times J_1,\ (j_{2,l,s},i_{2,l,s})\in J_3\times J_2,\ (j_{3,l,s},i_{3,l,s})\in J_3\times J_1,\\ &
(j_{4,l,s},i_{4,l,s})\in (n+ J_1)\times J_1,\ (j_{5,l,s},i_{5,l,s})\in (n+J_1)\times J_2,\ \\& (j_{6,l,s},i_{6,l,s})\in J_3\times(n+ J_3),
(j_{7,l,s},i_{7,l,s})\in J_2\times(n+ J_3),\,c_l\in \mathbb{F} ,\quad \ \\&
 \text{for each } l \in \ol{1,m}  ,\ r\in \ol{1,7},  \  s\in \ol{1,k_{r,l}}.
\end{split} \end{align}

Choose $0\neq v\in M_0$. It is homogeneous of ordinary degree $-\ell$. Every oscillator generator raises ordinary degree by at most two, so every element of $M_{p-1}$ has degree at most $-\ell+2p-2$. Since $\xi(v)\in M_{p-1}$, its component of degree $-\ell+2p$ is zero. This component is $h=\sum_lh_l$, where
\begin{align*} \begin{split}
h_l= (-1)^{k_{1,l} +k_{7,l}}c_l\Big(&
\prod_{s=1}^{k_{1,l}} x_{j_{1,l,s}}x_{i_{1,l,s}}
\prod_{s=1}^{k_{2,l}} y_{j_{2,l,s}}y_{i_{2,l,s}}
\prod_{s=1}^{k_{3,l}}( y_{j_{3,l,s}}y_{i_{3,l,s}} -x_{j_{3,l,s}}x_{i_{3,l,s}})\\
&
\prod_{s=1}^{k_{4,l}} (x_{j_{4,l,s}-n}y_{i_{4,l,s}}-y_{j_{4,l,s}-n}x_{i_{4,l,s}})
\prod_{s=1}^{k_{5,l}} x_{j_{5,l,s}-n}y_{i_{5,l,s}}\\
&
\prod_{s=1}^{k_{6,l}} (x_{i_{6,l,s}-n }y_{j_{6,l,s} }-x_{j_{6,l,s} }y_{i_{6,l,s}-n})
\prod_{s=1}^{k_{7,l}} x_{j_{7,l,s} }y_{i_{7,l,s}- n}
\Big)v .
\end{split}\end{align*}

Cancel the nonzero factor $v$ from $h=0$ and compare coefficients of monomials in $x_i,y_i$ ($i\in J_2$). We treat one such coefficient at a time. Within each group,
\begin{align} \begin{split}\label{a5.22}
& k_{1,1} +k_{7,1}=k_{1,l'} +k_{7,l'},\; \\ &\Big(
\prod_{s=1}^{k_{1,1}} x_{j_{1,1,s}} \Big) \Big(
\prod_{s=1}^{k_{7,1}} x_{j_{7,1,s}}
\Big)  =  \Big(
\prod_{s=1}^{k_{1,l'}} x_{j_{1,l',s}}\Big) \Big(
\prod_{s=1}^{k_{7,l'}} x_{j_{7,l',s}}
\Big), \\ & \Big(
\prod_{s=1}^{k_{2,1}}  y_{i_{2,1,s}}  \Big) \Big(
\prod_{s=1}^{k_{5,1}} y_{i_{5,1,s}}
\Big)  =   \Big(
\prod_{s=1}^{k_{2,l'}}y_{i_{2,l',s}}\Big) \Big(
\prod_{s=1}^{k_{5,l'}} y_{i_{5,l',s}}
\Big)
\end{split}\end{align}
for $l' \in \ol{2,m}$.

Set
\begin{align*} \begin{split} 
&h'=\sum_{l=1}^m c_l\Big[\Big(\prod_{s=1}^{k_{1,l}}  x_{i_{1,l,s}}\Big)\Big(\prod_{s=1}^{k_{5,l}} x_{j_{5,l,s}-n}\Big)
\Big(\prod_{s=1}^{k_{2,l}} y_{j_{2,l,s}} \Big)\Big(\prod_{s=1}^{k_{7,l}}  y_{i_{7,l,s}- n}\Big)\\
&
\cdot \Big(\prod_{s=1}^{k_{3,l}}z_{  j_{3,l,s}, i_{3,l,s}}  \Big)\Big(
\prod_{s=1}^{k_{4,l}}  \omega_{j_{4,l,s}-n,i_{4,l,s}}   \Big)\Big(
\prod_{s=1}^{k_{6,l}}  \omega_{i_{6,l,s}-n,j_{6,l,s} }
\Big) \Big]  \in \msr F_p,
\end{split}\end{align*}
The contribution of this group to $h$ is $\pm\phi(h')gv$, where $g$ is its fixed nonzero monomial in $\mathbb F[x_i,y_i:i\in J_2]$. Thus $\phi(h')=0$.
By Lemma \ref{lem:4.4}, we have  $h'\in \msr R \cap \msr F_p$.

Denote by $\ol{\xi} =\xi +U_{p-1}(\mfk g) \in S_{(p)} \left( \mathfrak{g} \right) $.
Recall the definition of the set $\mfk R$ (\ref{b4.29}) and consider the correspondence between the terms in $ h'g$ and the terms in $\ol{\xi}$.

\psp

The ten groups of relations, with the middle-variable factors written explicitly, are listed in Appendix~\ref{app:relations}. We also need to justify that polynomial coefficients in a relation can be lifted. Fix the multidegree $\gamma$ in the middle variables. Replacing a factor $a_i c_s$ by a symbol $Z_{is}$ has only the relations
$Z_{is}Z_{jt}-Z_{it}Z_{js}$: two assignments of the same $a$-factors to the same $c$-factors differ by successive pairwise exchanges. These are the minors in $I_{\Delta_2}$. The ideal $\msr R$ is homogeneous separately in the $a$-variables and in the alternating variables. A relation of type $ap$ can therefore be multiplied by one of the available $c_s$ and lifted; a relation of type $pp$ needs no such factor. The remaining $a$-factors can be assigned to the remaining middle-variable factors. This proves that the correspondence of generators also gives the required inclusion of ideals.

In conclusion, we have
\begin{align*}
\ol{\xi}  \in
 \left\langle \left. \ol{\Delta}^{n+j_1,s}_{i_1,t},\ol{\Delta}^{j_3,s}_{n+i_3,t}   \right\vert  i_1,j_1\in J_1 ; i_3,j_3 \in   J_3; (s,t) \in L \right\rangle+\langle I_{\Delta_2}\rangle .
\end{align*}

On the other hand, the operators in \eqref{a1.19}--\eqref{a1.20} vanish, as verified using \eqref{eq:ordered-pf}. Together with \eqref{5.15}, this gives the reverse inclusion.

Therefore, (\ref{5.20}) holds.
 \end{proof}
\pse
 
Define a bijective function $\theta: \{0, 1, \dots, 2n\} \to \{ 0,1, \dots, 2n\}$ by
\begin{align*}
\theta(i) = \begin{cases}
0 & \text{if } i =0, \\
n+i & \text{if } 1\leq i \leq n, \\
i-n & \text{if } i > n.
\end{cases}
\end{align*}
We extend this to pairs by setting $\theta(i, j) = (\theta(i), j)$ for $i, j \in \{1, \dots, 2n\}$.
 Then $\theta(L^c)=S_0$, with $S_0$ as in \eqref{eq:S0}.

Now we have our first main theorem:

\begin{theorem}\label{the:5.1}
 If $\ell   \leq 0$, we have
\begin{align}\label{5.43}
\begin{split}
\msr V({\msr H}_{\la\ell \ra})= \msr{V}  ( \la &\td{E}_{s',t'} , \Delta_{i,i'}^{j,j'} ,\ol{\Delta}^{n+j_1,s}_{i_1,t},\ol{\Delta}^{ j_3,s}_{n+i_3,t} \mid (s,t)\in L;(s',t')\in L^c;\\
&\theta(j),\theta(j')\in A;\ i,i'\in B ;  i_1,j_1\in J_1;i_3,j_3\in J_3\ra  ).
\end{split}\end{align}
Moreover,
\begin{align*}
\begin{split}
\msr V({\msr H}_{\la\ell \ra})\cong   \msr V_1(S_0) \cap {\bf Pf}(2n,2 ).
\end{split}\end{align*}
In particular, when $\ell \leq 0$, $n_1=n_2$ and $n \geq 3$, we have
\begin{align}\label{5.49}
 \mathbb P\msr V({\msr H}_{\la\ell \ra}) \cong \mathbf{Gr}(2,n)   .
 \end{align}

\end{theorem}

\begin{proof}
(1) Assume that $n_1<n_2$.

We use the same-block row and column condition in the definition of $I_{\Delta_2}$.

According to Lemma \ref{lem:5.2}, we have:
\begin{align*}\begin{split}
\msr V({\msr H}_{\la\ell \ra})= \msr{V}  ( \la &\td{E}_{s',t'} , \Delta_{i,i'}^{j,j'} ,\ol{\Delta}^{n+j_1,s}_{i_1,t},\ol{\Delta}^{ j_3,s}_{n+i_3,t} \mid (s,t)\in L;(s',t')\in L^c;\\
&\theta(j),\theta(j')\in A;\ i,i'\in B ;  i_1,j_1\in J_1;i_3,j_3\in J_3\ra  )  .
\end{split}\end{align*}

(2) Assume that $n_1=n_2=1$ and $n=2$. Then, writing $a_1=x_1$, $a_2=y_2$ and $w=y_1y_2-x_1x_2$, the harmonic space is
\[
\msr H_{\la\ell\ra}=\mathbb F[a_1,a_2]_{-\ell}\,\mathbb F[w].
\]
This follows either from the harmonic decomposition in \cite{LX}, or by solving $(x_1\partial_{y_1}+y_2\partial_{x_2})f=0$. Multiplication by $w$ raises the filtration by one; the other generators preserve or lower its $w$-degree. Consequently
\[
M_k=\bigoplus_{j=0}^k\mathbb F[a_1,a_2]_{-\ell}w^j.
\]
There is one nonzero coordinate in the associated variety and no quadratic relation on it. Thus \eqref{5.43} holds and the affine variety is $\mathbb A^1$.

(3) Assume that $n_1=n_2  $ and $n\geq 3$.\pse

Set \begin{align*}M_0=\operatorname{Span}\Big\{ \prod_{i=1}^{n_1} x_{i}^{\al_i}\prod_{j=n_1+1}^{n } y_{j}^{\be_j} \vert \sum_{i=1}^{n_1}\al_i + \sum_{j=n_1+1}^{n }\be_j =-\ell \Big\} .\end{align*}
 According to (7.2.50)-(7.2.52) in \cite{Xx}, we have $U(\mfk g)(M_0)={\msr H}_{\la\ell \ra}$.
Furthermore, we have
\begin{align}\label{5.47}
M_k=U_k(\mfk g)(M_0)=\sum_{j=0}^kM_0\cdot W^j
=\operatorname{Span}\Big\{f\prod_{i=1}^j w_i\mid f\in M_0,\ w_i\in W,\ 0\leq j\leq k
\Big\},\end{align}
where $W=W_1\cup W_3 \cup P$ (cf.(\ref{3.22})-(\ref{3.24})).
The summand $M_0W^j$ is homogeneous of ordinary degree $2j-\ell$, so these summands form a direct sum. The formula follows by multiplication with $W$ and the commutator argument of Lemma~\ref{lem:2.1}; the remaining generators preserve or lower these degrees. The highest degree argument in Lemma~\ref{lem:5.2} now applies with no middle variables.
According to (\ref{5.47}) and Lemma \ref{lem:4.4}, we can conclude that
\begin{align*}\begin{split}
&\msr V({\msr H}_{\la\ell \ra})\\= &\msr{V}  ( \la  \td{E}_{s',t'} ,  \ol{\Delta}^{n+j_1,s}_{i_1,t},\ol{\Delta}^{ j_3,s}_{n+i_3,t} \mid (s,t)\in L;(s',t')\in L^c;  i_1,j_1\in J_1;i_3,j_3\in J_3\ra  )  .
\end{split}\end{align*}

(4) When $ n_1=n_2=n$, we set
\begin{align*}
M_k= \operatorname{Span}\left\{x^\al\mid\sum_{i=1}^n\al_i=-\ell\right\}\cdot\sum_{j=0}^kW_1^j  \end{align*}
for $k\in \mbb N$.
Then $M_k= U_k(\mfk g)(M_0) $  and $U(\mfk g)(M_0)={\msr H}_{\la\ell \ra}$ by (7.2.53) in \cite{Xx}.
Moreover,\begin{align*} \msr V({\msr H}_{\la\ell \ra}) =  \msr{V}  ( \la  \td{E}_{s ,t } ,  \ol{\Delta}^{n+s_1,n+s_2}_{t_1,t_2}   \mid  s,s_1,t_1 , s_2,t_2 \in \ol{1,n} ; t\in \ol{1,2n}  \ra  )  .
 \end{align*}
In summary, (\ref{5.43}) holds if $n_1=n_2$.
This completes the proof of (\ref{5.43}).
\pse

Now define the algebra isomorphism $\Theta$ from $S(\mfk g)$ to $\mathbb F[\mathrm{Skew}(2n)]=\mathbb F[X_{ij}:1\leq i<j\leq2n]$, using $X_{ji}=-X_{ij}$ and $X_{ii}=0$, by $\Theta(\td{E}_{i,j}) =  X_{\theta(i),j }$.
For $i,j,s,t\in\ol{1,2n}$, we set
\begin{align} \label{b5.49}
\kappa_{i,t}^{j,s}=X_{ji }X_{st}-X_{js }X_{it}+X_{jt }X_{is}
 \end{align}
to be the Pfaffian of the corresponding principal $4\times4$ submatrix of $X\in\mathrm{Skew}(2n)$.
For $i_1,j_1\in J_1$, $i_3,j_3\in J_3$ and $(s,t)\in L$,
\begin{align*} 
\kappa_{i_1,t}^{j_1,\theta(s)}= \Theta\left( \ol{\Delta}^{n+j_1,s}_{i_1,t} \right),
\kappa_{i_3+n,t}^{j_3+n,\theta(s)}= \Theta\left( \ol{\Delta}^{ j_3,s}_{n+i_3,t} \right).
 \end{align*}
For $\theta(j),\theta(j')\in A$ and $i,i'\in B$, the missing Pfaffian term has its $B\times B$ coordinate equal to zero. Therefore
\begin{align}\label{b5.51}
\kappa_{i,i'}^{\theta(j),\theta(j')}=\Theta(\Delta_{i,i'}^{j,j'})
\quad\pmod{\langle X_{rs}:(r,s)\in S_0\rangle}.
\end{align}

On the other hand, we can verify that $ \Theta^{-1} (\kappa_{i,t}^{j,s})$ is always one of those generators (the only possible difference is a sign), or zero, modulo the zero-coordinate ideal.

Hence,
\begin{align*}
\begin{split}
\msr V({\msr H}_{\la\ell \ra}) \cong \msr V_1(\theta(L^c)) \cap {\bf Pf}(2n,2 )=\msr V_1(S_0) \cap {\bf Pf}(2n,2 ).
\end{split}\end{align*}

In particular, when $\ell \leq 0$, $n_1=n_2$ and $n \geq 3$, \begin{align*}
 \mathbb P\msr V({\msr H}_{\la\ell \ra}) \cong \mathbf{Gr}(2,n)   .
 \end{align*}
Consider the Pl\"ucker coordinates  $Q_{i,j}$ of $\mathbf{Gr}(2,n)$, which satisfy the classical Pl\"ucker relations:
\begin{align}\label{eq:plucker}
Q_{i_1,i_2}Q_{i_3,i_4} - Q_{i_1,i_3}Q_{i_2,i_4} + Q_{i_1,i_4}Q_{i_2,i_3} = 0,
\qquad (1 \le i_1 < i_2 < i_3 < i_4 \le n).
\end{align}

Let $A=J_1\cup(n+J_3)$ and let $\rho:\ol{1,n}\to A$ be the order-preserving bijection. When $n_1=n_2$, the zero coordinates in $S_0$ leave precisely the skew-symmetric block on $A$. The coordinate-ring isomorphism is
\begin{equation*}
\mathbb F[Q_{ij}:i<j]/(\text{Pl\"ucker relations})
\longrightarrow \mathbb F[\msr V(\msr H_{\la\ell\ra})],
\qquad Q_{ij}\longmapsto X_{\rho(i),\rho(j)}.
\end{equation*}
Under $\Theta$, its defining quadrics are exactly \eqref{eq:plucker}. Thus the affine associated variety is the Pl\"ucker cone over $\mathbf{Gr}(2,n)$; taking $\operatorname{Proj}$ gives \eqref{5.49}. The Grassmannian assertion is a direct consequence of the coordinate-section description.

\psp
\end{proof}
We have proved Theorem 1 stated in the Introduction.

\section[Extensions to the Odd Orthogonal Algebra]{Extensions to the Orthogonal Algebra $\mfk{o}(2n+1)$}

We now consider $M=\msr H'_{\la\ell\ra}$ and $\mfk g=\mfk o(2n+1)$, with $\ell\in\mathbb Z$. We first describe the filtration and then repeat the highest degree argument with the additional variable $x_0$. The filtration computation is kept separate from the even case.

The representation formulas are given by
\begin{align*}
\td{\pi}'(E_{i,j}-E_{n+j,n+i})=\td{\pi} (E_{i,j}-E_{n+j,n+i})\qquad \text{for } i,j\in\ol{1,n} \quad \text{(cf. \eqref{3.1}--\eqref{3.5})}
\end{align*}
and
\begin{align} \label{6.2}
\td{\pi}'(E_{0,i} )= \left\{\begin{array}{ll} -x_0 x_i &\text{if }
i \in\ol{1,n_1};\\  x_0 \ptl_{x_i} &\text{if } i\in\ol{n_1+1,n } ;\\
x_0 \ptl_{y_{i-n}} &\text{if } i\in\ol{n+1,n+n_2};\\
-x_0 y_{i-n} &\text{if } i \in\ol{n+n_2+1,2n};
\end{array}\right.
\end{align}
\begin{align} \label{6.3}
\td{\pi}'(E_{i,0} )= \left\{\begin{array}{ll} \ptl_{x_{0}}\ptl_{x_{i}} &\text{if }
i \in\ol{1,n_1};\\ x_i \ptl_{x_{0}} &\text{if } i\in\ol{n_1+1,n } ;\\
y_{i-n} \ptl_{x_{0}} &\text{if } i\in\ol{n+1,n+n_2};\\
\ptl_{x_{0}}\ptl_{y_{i-n}} &\text{if } i \in\ol{n+n_2+1,2n}.
\end{array}\right.
\end{align}
For convenience, we extend the definition of $\widetilde{E}_{i,j}$ to indices involving zero as follows:
\begin{align*}
\td{E}_{i,j} =\left\{\begin{array}{ll}
\td{E}_{i,j} \text{ in (\ref{1.15})}   &\mbox{if}\;
i,j\neq 0 ;\\
E_{0,j}-E_{j-n,0}  &\mbox{if}\;  i =0,j>n;\\
E_{0,j}-E_{n+j,0}  &\mbox{if}\;  i =0,j\leq n;\\
E_{i,0}-E_{0,i+n}   &\mbox{if}\;  j =0,i\leq n;\\
E_{i,0}-E_{0,i-n}  &\mbox{if}\;  j =0,i>n.
\end{array}\right.
\end{align*}

Let $J_0$ be the double integration operator with zero integration constants:
\[
J_0(x_0^rf)=\frac{x_0^{r+2}f}{(r+1)(r+2)}\qquad(f\in\msr A,\ r\geq0).
\]
On $\msr A\oplus x_0\msr A$, define
\begin{equation*}
T'=\sum_{i\geq0}(2J_0\widetilde\Delta)^i,\qquad
T'_\iota(f)=\sum_{i\geq0}\frac{2^i x_0^{2i+\iota}}{(2i+\iota)!}\widetilde\Delta^i(f)
\quad(\iota=0,1;\ f\in\msr A).
\end{equation*}
The sums are finite since $\widetilde\Delta$ is locally nilpotent. Thus $T'(f+x_0g)=T'_0(f)+T'_1(g)$. The inverse on the harmonic space takes a polynomial $h$ to $h|_{x_0=0}+x_0(\partial_{x_0}h)|_{x_0=0}$. For $i,j\neq0$, the even orthogonal operator $\widetilde E_{i,j}$ commutes with $\widetilde\Delta$ and acts on neither $x_0$ nor $J_0$. Hence
\begin{align} \label{6.5}
T' \td{E}_{i,j} =  \td{E}_{i,j} T' \quad \text{for } i,j \neq 0.
\end{align}
Equivalently, solving $\partial_{x_0}^2g=2\widetilde\Delta g$ coefficient by coefficient (cf. \cite{LX}), we obtain
\begin{align*}
{\msr H}_{\la \ell \ra}'=T'(\msr A_{\la \ell \ra} +\msr A_{\la \ell-1 \ra} x_0)=T_0'(\msr A_{\la \ell \ra} ) \oplus T_1'(\msr A_{\la \ell-1 \ra} ).
\end{align*}

Let ${\mfk d}_1' : {\msr A}'  \rightarrow \mathbb{Z}\cup \{-\infty\}$ be the weighted degree function (cf. \eqref{2.2}--\eqref{2.3}) defined by
\begin{align*} 
{\mfk d}_1' (x_0^{\alpha_0}x^\al y^\be)= \alpha_0+2 \sum_{i \in J_3} \alpha_i +\sum_{i \in J_2}\alpha_i +2\sum_{i \in J_1}\beta_i+\sum_{i \in J_2}\beta_i.
\end{align*}
For $ x^\al y^\be \in \msr A_{\la \ell \ra}$ and $x^{\al'} y^{\be'} \in \msr A_{\la \ell-1 \ra} $, we have
\begin{align*}
{\mfk d}_1'(T'_0(x^\al y^\be))={\mfk d}_1'(x^\al y^\be),\quad {\mfk d}_1'(T'_1(x^{\al'}y^{\be'}))=1+{\mfk d}_1'(x^{\al'}y^{\be'}).
\end{align*}
Indeed, each application of $\widetilde\Delta$ lowers the weight by two, and the accompanying factor $x_0^2$ raises it by two. The initial data cannot cancel. Thus ${\mfk d}_1'$ and $T'$ satisfy \eqref{2.6}.
Taking $P_1=W$, the fact that $T'(f\cdot p)=T'(f)p$ for $p\in W$ shows that $T'$ satisfies condition \eqref{2.7}.
For $f\in M$ and $(i,j) \notin L_0$, we have
\begin{align*} 
{\mfk d}_1' ( \td{E}_{i,j}(f) )\leq {\mfk d}_1' ( f )+1.
\end{align*}
Thus ${\mfk d}_1'$ satisfies condition \eqref{2.5}.

Hence, ${\mfk d}_1'$ and $T'$ satisfy the assumptions of Lemma \ref{lem:2.1}.

Denote
\begin{align} \label{6.10}
\ell'= \left\{\begin{array}{ll} 0 &\text{if } \ell \leq 0;\\ \ell &\text{if } \ell> 0,\,n_1<n_2;\\
2\ell-1 &\text{if } \ell > 0,\,n_1=n_2.
\end{array}\right.
\end{align}
In the rest of this section, $x^\alpha$ may include $x_0^{\alpha_0}$; in the initial-data space $\alpha_0\in\{0,1\}$. Define, for $k\in\mathbb Z$,
\begin{align*}
TN'(k)&=\operatorname{Span}\{T'(f): f\text{ is a monomial in }
\msr A_{\la\ell\ra}\oplus x_0\msr A_{\la\ell-1\ra},\ \mfk d'_1(f)=k+\ell'\},\\
TN'_k&=\sum_{i\leq k}TN'(i),\qquad
V'_k=\sum_{i=0}^kTN'_{k-i}W^i\quad(k\geq0).
\end{align*}
In particular, $V'_0=TN'_0$. It is finite-dimensional, since ordinary degree of a monomial of signed degree $\ell$ is at most $2\mfk d'_1-\ell$. It is nonzero: for $\ell\leq0$ it contains $x_1^{-\ell}$; for $\ell>0$, it contains $T'_0(x_{n_1+1}^{\ell})$ if $n_1<n_2$, and $T'_1(y_1^{\ell-1})$ if $n_1=n_2$.

\begin{lemma}\label{lem:6.1}
Assume $n_1=n_2$. Let $f\in\msr A_{\la\ell-1\ra}$ be a monomial and let $k\geq1$ satisfy $1+\mfk d'_1(f)=\ell'+k$. Then
\[
T'_1(f)=T'(x_0f)\in U_1(\mfk o(2n+1))(V'_{k-1}).
\]
\end{lemma}
\begin{proof}
When $\ell = 0$ and $k=1$, we have $V_0'=\mbb F$.
For $i_1\in J_1$ and $i_3\in J_3$,
\begin{align*}
\td{E}_{0,n+i_3}(-1)=x_0y_{i_3},\quad\td{E}_{0,i_1}(-1)=x_0x_{i_1}.\end{align*}
Hence $T'(x^\al y^\be)  \in U_1(\mfk{o}(2n+1))(V_{0}') $.

Now suppose that $\ell\neq0$ or $k\geq2$, and write $x^\alpha y^\beta=x_0f$.
Since
\begin{align*}
\ell-1+ \sum_{i=  1}^{n_1 } \al_i+\sum_{i=n_2+1}^{n } \be_i  = \sum_{i= n_1+1}^{n} \al_i+\sum_{i=1}^{n_2 } \be_i= \frac{\ell'+k -1}{2} ,  \end{align*}
we have $\sum_{i= 1}^{n_1} \al_i+\sum_{i=n_2+1}^{n } \be_i  > 0$ by \eqref{6.10}; the other displayed sum is only required to be nonnegative.
In fact the first sum equals $B+1-\ell$, where $B=(\ell'+k-1)/2$; it is positive for $\ell\leq0$, and equals $k/2>0$ for $\ell>0$. Without loss of generality, we may assume that $\sum_{i=1}^{n_1}\alpha_i>0$; the case $\sum_{i=n_2+1}^n \be_i>0$ is symmetric. Then we can write $x^\al y^\be = x_0 x_{i_1} g_1$ for some $i_1 \in J_1$ and $T'(g_1) \in TN'(k-1)$.

Since $\td{E}_{0,i_1}(T'( -g_1)) \in M=T'({\msr A}_{\la \ell \ra}+{\msr A}_{\la \ell-1 \ra}x_0)$, there exists $f' \in {\msr A}_{\la \ell \ra}+{\msr A}_{\la \ell-1 \ra}x_0$ such that $\td{E}_{0,i_1}(T'(- g_1)) =T'(f')$ and $\td{E}_{0,i_1}(T'(-g_1)) \equiv f' \mod \la x_0^2 \ra$.
According to \eqref{6.2}--\eqref{6.3},
\begin{align*} \begin{split}
\td{E}_{0,i_1}(T'(- g_1)) &\equiv(-x_0 x_{i_1} -y_{i_1} \ptl_{x_0})( - g_1 - x_0^2 \td{\Delta}(g_1) ) \mod \la x_0^2 \ra \\
&\equiv x_0 x_{i_1}g_1+2x_0 y_{i_1}  \td{\Delta}(g_1)  \mod \la x_0^2 \ra.
\end{split}\end{align*}
Thus
\begin{align}\label{a6.20}
T'(x^\al y^\be)  =  \td{E}_{0,i_1}(T'(- g_1))-T'\left(2x_0 y_{i_1}  \td{\Delta}(g_1)\right) . \end{align}

In the sums below, terms with $\beta_i=0$ or $\alpha_j=0$ are omitted before writing a quotient by $y_i$ or $x_j$. Thus all quotients denote polynomials. By \eqref{a6.20} and \eqref{1.8},
\begin{align} \begin{split} \label{a6.23}
T'(x^\al y^\be) =& \td{E}_{0,i_1}(T'(- g_1))
-T'\left(\sum_{i= 1}^{n_1}    \left( 2   \be_i  \frac{ x_0 x_i y_{i_1}g_1}{y_i}    \right) \right) \\
&- T'\left(\sum_{j=n_2+1}^n    \left( 2    \al_j  \frac{ x_0 y_j y_{i_1}g_1}{x_j}    \right) \right).
\end{split}\end{align}
According to \eqref{6.2}--\eqref{6.3},
\begin{align} \begin{split} \label{a6.24}
&T'\left(\sum_{j=n_2+1}^n    \left( 2    \al_j  \frac{ x_0 y_j y_{i_1}g_1}{x_j}    \right) \right)
\\=& \sum_{j=n_2+1}^n \left( \td{E}_{j,i_1}\left(T'\left(2    \al_j  \frac{ x_0 g_1}{x_j}   \right) \right) \right)+\left(2 \sum_{j=n_2+1}^n \al_j \right)T'(x^\al y^\be) .
\end{split}\end{align}
and
\begin{align} \begin{split} \label{a6.25}
&T'\left(\sum_{i= 1}^{n_1}    \left( 2    \be_i  \frac{ x_0 x_i y_{i_1}g_1}{y_i}    \right) \right)
\\=& \sum_{i= 1}^{n_1}  \left( \td{E}_{n+i,i_1}\left(T'\left(2    \be_i  \frac{ x_0 g_1}{y_i}   \right) \right) \right)+\left(2 \sum_{i= 1}^{n_1} \be_i \right)T'(x^\al y^\be).
\end{split}\end{align}
Combining \eqref{a6.23}, \eqref{a6.24} and \eqref{a6.25} yields
\begin{align*} \begin{split}
&\left(1+2 \sum_{j=n_2+1}^n \al_j  +2 \sum_{i= 1}^{n_1} \be_i \right)T'(x^\al y^\be) \\
 \quad =&\td{E}_{0,i_1}(T'(- g_1))- \sum_{j=n_2+1}^n \left( \td{E}_{j,i_1}\left(T'\left(2    \al_j  \frac{ x_0 g_1}{x_j}   \right) \right) \right)\\
&-\sum_{i= 1}^{n_1}  \left( \td{E}_{n+i,i_1}\left(T'\left(2    \be_i  \frac{ x_0 g_1}{y_i}   \right) \right) \right)  .
\end{split}\end{align*}

For $i\in J_1$ such that $\be_i >0$, we have $\mfk d'_1\left( \frac{ x_0 g_1}{y_i}  \right) =\mfk d'_1(x^\al y^\be)-2$, and so $T'\left( \frac{ x_0 g_1}{y_i}  \right)\in V_{k-1}'$.
For $j\in J_3$ such that $\al_j >0$, we have $\mfk d'_1\left( \frac{ x_0 g_1}{x_j}  \right) =\mfk d'_1(x^\al y^\be)-2$, and so $T'\left( \frac{ x_0 g_1}{x_j}  \right)\in V_{k-1}'$.
Therefore, we have
\begin{align*}   T'(x^\al y^\be) \in \, U_1(\mfk{o}(2n+1))(V_{k-1}'). \end{align*}
\end{proof}

\begin{proposition}\label{prop:odd-filtration}
For $k\in \mbb N$, let $M_0=V_0'$. Then $M_k=U_k(\mfk g)(V_0')=V_k'$. Moreover, $M={\msr H}'_{\la\ell\ra}=U(\mfk{o}(2n+1))(V_0')$.
\end{proposition}
\begin{proof}
The degree estimates above and Lemma~\ref{lem:2.1} give $M_k\subseteq V'_k$. For the reverse inclusion, it remains to lower a monomial in $TN'(k)$ to $V'_{k-1}$. If it contains an outer positive-weight variable, the operators involving index zero give this directly, or by the finite descending induction in Appendix~\ref{app:odd-filtration}. The middle-variable cases use the even generators. Lemma~\ref{lem:6.1} handles the remaining odd initial data when $n_1=n_2$. The full case calculation, including the weight-one boundary case, is given in that appendix. It proves $TN'(k)\subseteq U_1(\mfk g)(V'_{k-1})$ for $k\geq1$. Lemma~\ref{lem:2.1} now gives equality. Finally, the initial-data expansion shows $M=\bigcup_k TN'_k=\bigcup_k M_k$.
\end{proof}
\pse

Set (cf.(\ref{1.17}))
\begin{align}\label{a1.29}
L' =L\cup (\{0\}\times J_1)\cup (J_3 \times\{0\} )\cup (\{0\}\times (n+J_3))\cup ((n+J_1) \times\{0\} ).\end{align} 
and $(L')^c =\{0, \dots, 2n\}^2  \setminus \{(0,0)\} \setminus L'$.

Note that $\mfk d'_1(\td{E}_{j,i} (f)) \leq \mfk d'_1(f)$ for $(j,i) \in (L')^c $ and $f\in M$ (cf. \eqref{a1.29}).
By Lemma~\ref{lem:5.1}, this gives the inclusion from right to left below. For the converse, choose an element of $M_0$ of maximal ordinary degree and compare the independent quadratic multiplication parts, as in Section 5. Thus
\begin{align*}
I_{(1)}=\operatorname{Span}\{\td{E}_{j,i}  \mid  (j,i)  \in (L')^c\}.
\end{align*}

We extend the parameter range of $\Delta_{i_1,i_2}^{\,j_1,j_2}$ in equation \eqref{a1.18} from $i_1,i_2, j_1,j_2 \in \ol{1,2n}$ to $i_1,i_2, j_1,j_2 \in \ol{0,2n} $, while keeping the defining form unchanged. As an element of the symmetric algebra, this determinant vanishes if its two row indices or its two column indices coincide. No vanishing convention is imposed for coincidences between a row and a column index.
For instance, $\Delta_{i_1,0}^{\,j_1,j_2} = \td{E}_{j_1,i_1} \td{E}_{j_2,0} - \td{E}_{j_1,0} \td{E}_{j_2,i_1}$.

Similarly, using the convention \eqref{a1.18}, we extend the domain of $(s,t)$ in Equations \eqref{a1.19} and \eqref{a1.20} from $L$ to $L'$, supplemented by the definitions
\begin{align}   \label{a1.30}
\ol{\Delta}^{n+j_1,s}_{i_1,0} =
 \td{E}_{s,0}\td{E}_{n+j_1,i_1} - \td{E}_{s, j_1}\td{E}_{0,i_1} +\td{E}_{s,i_1}\td{E}_{0,j_1}   ,
\end{align} and
  \begin{align}     \label{a1.31}
\ol{\Delta}^{ j_3,0}_{n+i_3,t} =
  \td{E}_{0,t}\td{E}_{ j_3,n+i_3}-\td{E}_{ j_3,0}\td{E}_{ i_3,t}+\td{E}_{ j_3,t} \td{E}_{i_3,0 }  .
\end{align}

Set
\begin{align*}
I'_{\Delta_2}= \{ \Delta_{i_1,i_2}^{j_1,j_2} \mid (j_1,i_1),(j_2,i_2),(j_2,i_1),(j_1,i_2)\in L'\setminus L_0 \} .\end{align*}
Under $\theta$, these minors have row indices in $A$ and column indices in $B\cup\{0\}$, or the transposed arrangement. When $n_1=n_2$, the latter index set is $\{0\}$, so $I'_{\Delta_2}=\{0\}$ in the symmetric algebra.

\begin{lemma}\label{lem:6.2}
For $\Delta_{i_1,i_2}^{j_1,j_2}\in I'_{\Delta_2}$,
\begin{equation}\label{a6.39}
\left(\Delta_{i_1,i_2}^{j_1,j_2}\right)^{\ell'+1}
\in I_{(2\ell'+2)}.
\end{equation}
In fact, $I'_{\Delta_2}\subseteq I_{(2)}$.
\end{lemma}
\begin{proof}
Let $v_0\in M_0$. For $a,b\in J_1$, put $W_{ab}=x_a y_b-x_b y_a$. For $c,d\in J_2$, direct calculation gives
\[
\Delta_{a,b}^{c,d}=W_{ab}(x_c\partial_{y_d}-x_d\partial_{y_c}).
\]
The operator in parentheses preserves $M_0$, and multiplication by $W_{ab}$ sends $M_0$ into $M_1$. Hence this minor sends $M_0$ into $M_1$. Here and below we may use transposed minors, which give the same elements of $S(\mfk g)$ up to sign.

For the two cases involving the index zero, we have
\begin{align*}
\Delta_{c,0}^{n+b,n+a}
&=W_{ab}(x_0\partial_{x_c}-y_c\partial_{x_0}),\\
\Delta_{a,b}^{c,0}
&=W_{ab}(x_c\partial_{x_0}-x_0\partial_{y_c}).
\end{align*}
Write $v_0=T'(f+x_0g)$. Taking the value and first $x_0$-derivative at zero, and using $\partial_{x_0}^2v_0=2\widetilde\Delta(v_0)$, gives
\begin{align*}
\Delta_{c,0}^{n+b,n+a}(v_0)
&=W_{ab}T'\bigl(-y_cg+x_0(\partial_{x_c}f-2y_c\widetilde\Delta(f))\bigr),\\
\Delta_{a,b}^{c,0}(v_0)
&=W_{ab}T'\bigl(x_cg+x_0(2x_c\widetilde\Delta(f)-\partial_{y_c}f)\bigr).
\end{align*}
The corresponding statements for $J_3$ follow by the same calculation.

To include all the index choices, use \eqref{eq:block-minors} with $c,d\in B\cup\{0\}$. Each $F_{cd}$ preserves $M_0$: it commutes with the transformed Laplace operator, preserves the signed degree, and preserves the bound defining $TN'_0$, since all the variables indexed by $B\cup\{0\}$ have $\mfk d'_1$-weight one. Thus both $F_{cd}$ and $F_{cd}-g_{cd}$ preserve $M_0$. The factor $F_{ab}$ sends $M_0$ into $M_1$. Every minor therefore sends $M_0$ into $M_1$, and Lemma~\ref{lem:5.1} gives $I'_{\Delta_2}\subseteq I_{(2)}$. Taking the $(\ell'+1)$-st power proves \eqref{a6.39}.
\end{proof}

\begin{lemma}\label{lem:6.3}
Let $J'$ be generated by $I_{(1)}$, $I'_{\Delta_2}$ and $\overline\Delta^{n+j_1,s}_{i_1,t}$, $\overline\Delta^{j_3,s}_{n+i_3,t}$, with $i_1,j_1\in J_1$, $i_3,j_3\in J_3$ and $(s,t)\in L'$. Then
\begin{equation}\label{a6.47}
I_{(p)}=J'\cap S_{(p)}(\mfk g)\qquad(p\geq1).
\end{equation}
\end{lemma}

\begin{proof}
We use the seven groups in \eqref{a5.20} together with the two groups involving $x_0$. Work modulo $\langle I_{(1)}\rangle$.

Let $\xi \in \mfk k_{(p)}$ and write
\begin{align*}
\xi =  \sum_{l=1}^m c_l\left(
\prod_{r=1}^7 \left(\prod_{s=1}^{k_{r,l}} \td{E}_{j_{r,l,s},i_{r,l,s}}\right)
\right)
\left(\prod_{s=1}^{k_{8,l}} \td{E}_{j_{8,l,s},0}\right)
\left(\prod_{s=1}^{k_{9,l}} \td{E}_{j_{9,l,s},0}\right)
 ;
\end{align*}
where $m,k_{1,l},\dots,k_{9,l} \in\mbb N$ satisfy $\sum_{r=1}^9k_{r,l} =p$, the range of $(j_{r,l,s},i_{r,l,s})$ is the same as that in equations \eqref{a5.19}--\eqref{a5.20} when $r=1,\dots,7$ and
\begin{align*}\begin{split} j_{8,l,s}\in (n+J_1) ,\,j_{9,l,s'}\in J_3 ,\,c_l\in \mathbb{F} ,\quad \
 \text{for each } l \in \ol{1,m}  , \ s\in \ol{1,k_{8,l}}, \ s'\in \ol{1,k_{9,l}}.
\end{split} \end{align*}

Let $d_0$ be the maximal ordinary degree of an element of $M_0$ and choose $v\in M_0$ with a nonzero degree-$d_0$ component $v_{d_0}$. Since every generator raises ordinary degree by at most two, $M_{p-1}$ has degree at most $d_0+2p-2$. Therefore the degree-$d_0+2p$ component $h=\sum_lh_l$ of $\xi(v)$ is zero, where
\begin{align*} \begin{split}
h_l= (-1)^{k_{1,l} +k_{7,l}}&c_l\Big(
\prod_{s=1}^{k_{1,l}} x_{j_{1,l,s}}x_{i_{1,l,s}}
\prod_{s=1}^{k_{2,l}} y_{j_{2,l,s}}y_{i_{2,l,s}}
\prod_{s=1}^{k_{3,l}}( y_{j_{3,l,s}}y_{i_{3,l,s}} -x_{j_{3,l,s}}x_{i_{3,l,s}})\\
&
\prod_{s=1}^{k_{4,l}} (x_{j_{4,l,s}-n}y_{i_{4,l,s}}-y_{j_{4,l,s}-n}x_{i_{4,l,s}})
\prod_{s=1}^{k_{5,l}} x_{j_{5,l,s}-n}y_{i_{5,l,s}}\prod_{s=1}^{k_{7,l}} x_{j_{7,l,s} }y_{i_{7,l,s}- n}\\
&
\prod_{s=1}^{k_{6,l}} (x_{i_{6,l,s}-n }y_{j_{6,l,s} }-x_{j_{6,l,s} }y_{i_{6,l,s}-n})
\prod_{s=1}^{k_{8,l}} x_{j_{8,l,s} -n}x_{0}\prod_{s=1}^{k_{9,l}} y_{j_{9,l,s} }x_{0}
\Big)v_{d_0}.
\end{split}\end{align*}
Cancel $v_{d_0}\neq0$ and separate the monomials in $x_0$ and in the middle variables. In each group, \eqref{a5.22} holds and $k_{8,1}+k_{9,1}=k_{8,l'}+k_{9,l'}$ for $2\leq l'\leq m$.
Set
\begin{align*} \begin{split}
&h'=\sum_{l=1}^m c_l\Big[\Big(\prod_{s=1}^{k_{1,l}}  x_{i_{1,l,s}}\Big)\Big(\prod_{s=1}^{k_{5,l}} x_{j_{5,l,s}-n}\Big)
\Big(\prod_{s=1}^{k_{2,l}} y_{j_{2,l,s}} \Big)\Big(\prod_{s=1}^{k_{7,l}}  y_{i_{7,l,s}- n}\Big)
 \Big(\prod_{s=1}^{k_{3,l}}z_{  j_{3,l,s}, i_{3,l,s}}  \Big)\\
&
\cdot \Big(
\prod_{s=1}^{k_{4,l}}  \omega_{j_{4,l,s}-n,i_{4,l,s}}   \Big)\Big(
\prod_{s=1}^{k_{6,l}}  \omega_{i_{6,l,s}-n,j_{6,l,s} }
\Big) \Big(
\prod_{s=1}^{k_{8,l}}  x_{j_{8,l,s}-n }
\Big) \Big(
\prod_{s=1}^{k_{9,l}}  y_{j_{9,l,s} }
\Big)
\Big]  \in \msr F_p,
\end{split}\end{align*}
For each group, its contribution is $\pm\phi(h')gv_{d_0}$, with $g$ a nonzero monomial in $x_0$ and the middle variables. Thus $\phi(h')=0$.
By Lemma \ref{lem:4.4}, we have  $h'\in \msr R \cap \msr F_p$.

Denote $\ol{\xi} =\xi +U_{p-1}(\mfk g) \in S_{(p)} \left( \mathfrak{g} \right) $.
The five additional groups are listed at the end of Appendix~\ref{app:relations}. The coefficient-lifting argument in Lemma~\ref{lem:5.2} applies with $x_0$ included among the middle variables. Hence the vanishing relation $h'$ lifts to $J'$.

In conclusion, we have
\begin{align*}
\ol{\xi}  \in
 \left\langle \left. \ol{\Delta}^{n+j_1,s}_{i_1,t},\ol{\Delta}^{j_3,s}_{n+i_3,t}   \right\vert  i_1,j_1\in J_1 ; i_3,j_3 \in   J_3; (s,t) \in L' \right\rangle+\langle I'_{\Delta_2}\rangle .
\end{align*}

For $(s,t)\in L'$, the operators in \eqref{a1.19}--\eqref{a1.20} and their extensions \eqref{a1.30}--\eqref{a1.31} vanish. This follows from \eqref{eq:ordered-pf}, including the cases involving the index zero. Together with Lemma~\ref{lem:6.2}, this gives the converse inclusion. Therefore, \eqref{a6.47} holds.
\end{proof}

The following theorem corresponds to Theorem 2 as stated in the Introduction.

\begin{theorem}\label{the:odd}
For $\mfk{g} = \mathfrak{o}(2n+1)$ and $\ell \in \mbb{Z}$, the variety $\msr{V}(\msr{H}'_{\la\ell \ra})$ is given by
\begin{align}\label{a6.60}
\begin{split}
\msr{V}(\msr{H}'_{\la\ell \ra}) = \msr{V} \Big( \big\la &\td{E}_{s',t'} , \Delta_{i,i'}^{j,j'} ,\ol{\Delta}^{n+j_1,s}_{i_1,t},\ol{\Delta}^{ j_3,s}_{n+i_3,t} \mid (s,t)\in L', (s',t')\in (L')^c; \\
&\theta(j),\theta(j')\in A;\ i,i'\in B\cup\{0\}; i_1,j_1\in J_1; i_3,j_3\in J_3 \big\ra \Big).
\end{split}
\end{align}
Moreover,
\begin{align} \label{b6.62}
\msr{V}(\msr{H}'_{\la\ell \ra}) \cong \msr{V}_1(S_0') \cap {\bf Pf}(2n+1, 2).
\end{align}
\end{theorem}

\begin{proof}
By Lemmas \ref{lem:6.2} and \ref{lem:6.3}, the identity \eqref{a6.60} holds.

Note that $\theta((L')^c)=S_0'$.
Analogous to the proof of Theorem \ref{the:5.1}, we define the homomorphism $\Theta'$ from $S(\mfk{o}(2n+1))$ to $\mathbb F[\mathrm{Skew}(2n+1)]=\mathbb F[X_{ij}:0\leq i<j\leq2n]$, with $X_{ji}=-X_{ij}$ and $X_{ii}=0$, by setting $\Theta'(\td{E}_{i,j}) = X_{\theta(i),j}$. For $i,j,s,t \in \{0, \dots, 2n\}$, we define
\begin{align*}
\kappa_{i,t}^{j,s} = X_{ji}X_{st} - X_{js}X_{it} + X_{jt}X_{is}
\end{align*}
as the Pfaffian polynomial associated with the $4 \times 4$ principal submatrices of a matrix $X = (X_{ij}) \in \mathrm{Skew}(2n+1)$.

Observing that the relations \eqref{b5.49}--\eqref{b5.51} remain valid, we have
\begin{align*}
\kappa_{i_1,0}^{j_1,\theta(s)} = \Theta'\left( \ol{\Delta}^{n+j_1,s}_{i_1,0} \right), \quad
\kappa_{i_3+n,t}^{j_3+n,0} = \Theta'\left( \ol{\Delta}^{ j_3,0}_{n+i_3,t} \right)
\end{align*}
for $i_1, j_1 \in J_1$, $i_3, j_3 \in J_3$, and $(s,0), (0,t) \in L'$.

Furthermore, it can be verified that $(\Theta')^{-1}(\kappa_{i,t}^{j,s})$ is always either one of the aforementioned generators (possibly differing by a sign) or zero, provided that $X_{ij} = 0$ for all $(i,j) \in \theta((L')^c)=S_0'$. Consequently, \eqref{b6.62} holds.
\end{proof}

This completes the proof of Theorem 2 stated in the introduction.

\psp

\psp

\appendix
\section{The relations used in the highest degree argument}\label{app:relations}
We give the ten groups used in Lemma~\ref{lem:5.2}. This is the same grouping by $z$, $\omega$, $x_i$ and $y_j$ as in Section 4. The factors $c$ below are any of the middle variables $x_s,y_s$ with $s\in J_2$. All relations are considered modulo the zero coordinates. A relation is lifted by replacing each product with the corresponding oscillator symbol from \eqref{a5.20}; different assignments of the middle factors are handled in group (10).

(1) For $r\in J_3$ and $i,j,t\in J_1$,
\[
z_{r,i}\omega_{t,j}-z_{r,j}\omega_{t,i}+z_{r,t}\omega_{j,i}=0.
\]
This is the relation $\Lambda(z_{r,i}\omega_{t,j})$ and lifts to $\overline\Delta^{n+t,r}_{j,i}$, up to the sign determined by the order of the two alternating indices.

(2) For $i\in J_1$ and $r,s,t\in J_3$,
\[
z_{r,i}\omega_{t,s}-z_{t,i}\omega_{r,s}+z_{s,i}\omega_{r,t}=0.
\]
It is $\Lambda(z_{r,i}\omega_{t,s})$ and gives the second family of $\overline\Delta$.

(3) For $i,j\in J_1$ and $r,s\in J_3$,
\[
\omega_{s,r}\omega_{j,i}+z_{s,j}z_{r,i}-z_{s,i}z_{r,j}=0.
\]
This is the mixed $\omega_1\omega_3$ relation. In Pfaffian coordinates its four indices consist of two from $J_1$ and two from $n+J_3$.

(4) For $i,j,s,t\in J_1$,
\[
\omega_{t,j}\omega_{s,i}-\omega_{t,i}\omega_{s,j}+\omega_{t,s}\omega_{i,j}=0.
\]
All four coordinate indices lie in $J_1$. These are the relations between the $W_1$ factors.

(5) For $i,j,s,t\in J_3$,
\[
\omega_{t,j}\omega_{s,i}+\omega_{t,s}\omega_{i,j}-\omega_{t,i}\omega_{s,j}=0.
\]
All four coordinate indices lie in $n+J_3$. These are the relations between the $W_3$ factors.

(6) For $i,j,t\in J_1$, multiply the relation involving the extra column by $c$:
\[
(x_i\omega_{t,j}-x_j\omega_{t,i}+x_t\omega_{j,i})c=0.
\]
When $c=y_s$, it lifts, with the same alternating-index convention, to $\overline\Delta^{n+t,n+i}_{j,s}$; when $c=x_s$, it lifts to $-\overline\Delta^{n+t,s}_{j,i}$.

(7) For $r\in J_3$ and $i,j\in J_1$, the relation is
\[
(y_r\omega_{j,i}-x_jz_{r,i}+x_iz_{r,j})c=0.
\]
Its lift has one middle coordinate and the three active coordinates $i,j,n+r$. This includes the relation originally written using $\omega_{n+1,r}=-y_r$.

(8) For $i\in J_1$ and $r,t\in J_3$, we have
\[
(x_i\omega_{t,r}-y_tz_{r,i}+y_rz_{t,i})c=0.
\]
Its four Pfaffian indices are $i,n+r,n+t$ and the chosen middle coordinate.

(9) For $r,s,t\in J_3$, the relation is
\[
(y_r\omega_{t,s}-y_t\omega_{r,s}+y_s\omega_{r,t})c=0.
\]
It gives the remaining three-active-one-middle Pfaffians.

(10) Finally, the assignments of middle factors give two-term minors. If $a_i,a_j$ denote any of the variables $x_i$ ($i\in J_1$), $y_i$ ($i\in J_3$), and $c_s,c_t$ are middle variables, then
\[
(a_ic_s)(a_jc_t)-(a_ic_t)(a_jc_s)=0.
\]
These include mixed $x_i,y_j$ pairs as well as pairs from the same index set. In skew coordinates the two row indices are in $A$ and the two column indices are in $B$. The omitted Pfaffian term contains a $B\times B$ coordinate. This is the block condition used in the definition of $I_{\Delta_2}$. Transposing the block gives the same relations up to sign, so one orientation suffices.

The five additional groups in Lemma~\ref{lem:6.3} are obtained by setting $c=x_0$ in groups (6), (7), (8), (9), and by allowing $x_0$ as either middle factor in group (10). Explicitly, the four three-active-one-middle relations are
\begin{align*}
&(x_i\omega_{t,j}-x_j\omega_{t,i}+x_t\omega_{j,i})x_0=0,\\
&(y_r\omega_{j,i}-x_jz_{r,i}+x_iz_{r,j})x_0=0,\\
&(x_i\omega_{t,r}-y_tz_{r,i}+y_rz_{t,i})x_0=0,\\
&(y_r\omega_{t,s}-y_t\omega_{r,s}+y_s\omega_{r,t})x_0=0.
\end{align*}
They lift to the quadrics involving index zero in \eqref{a1.30}--\eqref{a1.31}. The fifth group is
\[
(a_ic_s)(a_jx_0)-(a_ix_0)(a_jc_s)=0.
\]
The coefficient-lifting argument at the end of Lemma~\ref{lem:5.2} proves that these lists account for arbitrary polynomial relations, not just for the displayed generators.

\section{The odd filtration calculation}\label{app:odd-filtration}
Here we give the full lowering calculation for Proposition~\ref{prop:odd-filtration}. The notation, the operators and the order of the cases are those of Section 6.
\begin{proof}
Take $\mfk g=\mfk{o}(2n+1)$, $\msr C=\msr A'$ , $\msr C_0={\msr A}_{\la \ell \ra}+{\msr A}_{\la \ell-1 \ra}x_0$, $\Gamma=\mbb Z$, $k_0=\ell'$, $T=T'$ and $P_1=W$ in Lemma \ref{lem:2.1}. We have verified that ${\mfk d}'_1$, $T'$ and $V_{k,\ell'}=V'_k$ satisfy the hypotheses of Lemma \ref{lem:2.1}. Thus we have $M_k \subseteq V'_k$ for $k\in \mbb N$ by Lemma \ref{lem:2.1}.

For an integer $k \geq 1$, we take $T'(x^\al y^\be) \in TN'(k)$, where $x^\al y^\be \in  \msr  A_{\la \ell \ra}+\msr A_{\la \ell-1 \ra}x_0 $.
Since $k\geq 1$, we have
\begin{align*} \begin{split}
\ell-\al_0+ \sum_{i=  1}^{n_1 } \al_i+\sum_{i=n_2+1}^{n } \be_i  =& \sum_{i= n_1+1}^{n} \al_i+\sum_{i=1}^{n_2 } \be_i \\>& \ell'-\al_0-\sum_{i= n_2+ 1}^{n } \al_i-\sum_{i= 1}^{n_1 } \be_i .
\end{split}\end{align*}
When $ \sum_{i=n_2+1}^n \al_i=\sum_{i=1}^{n_1} \be_i =0$, we have $ \sum_{i=0}^{n_1 } \al_i>0 $ or $\sum_{i=n_2+1}^{n } \be_i>0$. In particular, $\sum_{i=n_2+1}^n \al_i+\sum_{i=1}^{n_1} \be_i  > 0$  if $n_1=n_2$ and $\al_0=0$.

Next we verify that $T'(x^\al y^\be) \in  U_1(\mfk{o}(2n+1))( V_{k-1}')$ so that \eqref{2.13} holds. We consider the following ten cases. The calculations use the initial-data identity $T'(f)\equiv f\pmod{\langle x_0^2\rangle}$ and its uniqueness.
\psp

(1) When $\al_0=0$ and $\sum_{i=1}^{n_1} \be_i>0$, we write $x^\al y^\be=y_{i_1} g_1$, where $i_1\in J_1$.\pse

Since $\td{E}_{0,i_1}(T'(-x_0g_1)) \in M=T'({\msr A}_{\la \ell \ra}+{\msr A}_{\la \ell-1 \ra}x_0)$, there exists $f' \in {\msr A}_{\la \ell \ra}+{\msr A}_{\la \ell-1 \ra}x_0$ such that $\td{E}_{0,i_1}(T'(-x_0g_1)) =T'(f')$ and $\td{E}_{0,i_1}(T'(-x_0g_1)) \equiv f' \mod \la x_0^2 \ra$.
According to \eqref{6.2}--\eqref{6.3},
\begin{align*} \begin{split}
\td{E}_{0,i_1}(T'(-x_0g_1)) &\equiv(-x_0 x_{i_1} -y_{i_1} \ptl_{x_0})(-x_0g_1) \mod \la x_0^2 \ra \\
&\equiv y_{i_1}g_1\mod \la x_0^2 \ra.
\end{split}\end{align*}
Thus $T'(x^\al y^\be)=\td{E}_{0,i_1}(T'(-x_0g_1))$.
Note that $\mfk d'_1(-x_0g_1) =\mfk d'_1(y_{i_1}g_1)-1=k-1+\ell' $, and so $T'(-x_0g_1)\in V_{k-1}'$. Then
$T'(x^\al y^\be) \in U_1(\mfk{o}(2n+1))( V_{k-1}')$.\pse

(2) When $\al_0=0$ and $\sum_{i=n_2+1}^n \al_i>0$, we write $x^\al y^\be=x_{i_3} g_1$, where $i_3\in J_3$.\pse

According to \eqref{6.2}--\eqref{6.3},
\begin{align*} \begin{split}
\td{E}_{0,n+i_3}(T'(-x_0g_1)) &\equiv(-x_0 y_{i_3} -x_{i_3} \ptl_{x_0})(-x_0g_1) \mod \la x_0^2 \ra \\
&\equiv x_{i_3} g_1\mod \la x_0^2 \ra.
\end{split}\end{align*}
Then $T'(x^\al y^\be)=\td{E}_{0,n+i_3}(T'(-x_0g_1))$ and we have $T'(x^\al y^\be) \in U_1(\mfk{o}(2n+1))( V_{k-1}')$.\pse

(3) Assume $\al_0>0$, $n_1<n_2$, and $\sum_{i=1}^{n_1} \be_i>0$.\pse
 
 Under these conditions, we can write $x^\al y^\be = x_0 y_{i_1} g_1$ for some $i_1 \in J_1$.
From \eqref{6.2}--\eqref{6.3} we obtain
\begin{align*} \begin{split}
\td{E}_{0,i_1}(T'(x_{n_2} y_{n_2} g_1)) &\equiv(-x_0 x_{i_1} -y_{i_1} \ptl_{x_0})(x_{n_2} y_{n_2} g_1+
 x_0^2 \td{\Delta} (x_{n_2} y_{n_2} g_1)
) \mod \la x_0^2 \ra\\
&\equiv   -x_0 x_{i_1}x_{n_2} y_{n_2} g_1-2x_0 y_{i_1}\td{\Delta} (x_{n_2} y_{n_2} g_1) \mod \la x_0^2 \ra
\end{split}\end{align*}
Consequently,
\begin{align}\label{6.21} \begin{split} \td{E}_{0,i_1}(T'(x_{n_2} y_{n_2} g_1))
=&T'\left(-x_0 x_{i_1}x_{n_2} y_{n_2} g_1\right)- T'\left(2x_0 y_{i_1}\td{\Delta} (x_{n_2} y_{n_2} g_1) \right)  \\
=&-T'\left(x_0 x_{i_1}x_{n_2} y_{n_2} g_1\right)+2(\al_{n_2}+1)(\be_{n_2}+1)T'(x^\al y^\be)\\
& \quad-T'\left(2x_0 y_{i_1}(\td{\Delta}+\ptl_{x_{n_2}}\ptl_{y_{n_2}}) (x_{n_2} y_{n_2} g_1) \right).   \end{split} \end{align}
Using \eqref{6.5}, we also have
\begin{align} \begin{split}\label{6.22}
\td{E}_{n_2,i_1}(T'(x_{0} y_{n_2} g_1)) &=T'\left((-x_{n_2} x_{i_1} -y_{i_1} \ptl_{y_{n_2}})(  x_{0} y_{n_2} g_1)\right)\\
&=-T'\left( x_{n_2} x_{i_1}  x_{0} y_{n_2} g_1 \right)- (\be_{n_2}+1)T'(x^\al y^\be).
\end{split}\end{align}
Combining \eqref{6.21} and \eqref{6.22} yields
\begin{align} \begin{split} \label{6.23}
& (2\al_{n_2}+3)(\be_{n_2}+1)T'(x^\al y^\be)\\
 = &\td{E}_{0,i_1}(T'(x_{n_2} y_{n_2} g_1))-\td{E}_{n_2,i_1}(T'(x_{0} y_{n_2} g_1))
\\&+ 2 T'\left( x_0 y_{i_1}(\td{\Delta}+\ptl_{x_{n_2}}\ptl_{y_{n_2}}) (x_{n_2} y_{n_2} g_1) \right).
\end{split}\end{align}
Observe that $\mfk d'_1(x_{n_2} y_{n_2} g_1)=\mfk d'_1(x_{0} y_{n_2} g_1)=\mfk d'_1(x^\al y^\be)-1$. If $(\td{\Delta}+\ptl_{x_{n_2}}\ptl_{y_{n_2}}) (x_{n_2} y_{n_2} g_1)=0$, then \eqref{6.23} immediately gives $T'(x^\al y^\be)  \in U_1(\mfk{o}(2n+1))( V_{k-1}')$. We may therefore assume $(\td{\Delta}+\ptl_{x_{n_2}}\ptl_{y_{n_2}}) (x_{n_2} y_{n_2} g_1)\neq 0$.

In view of \eqref{6.23}, it suffices to prove that $T'(h) \in U_1(\mfk{o}(2n+1))( V_{k-1}')$ for every monomial $h$ appearing in $x_0 y_{i_1}(\td{\Delta}+\ptl_{x_{n_2}}\ptl_{y_{n_2}}) (x_{n_2}y_{n_2}g_1)$.
Define a weight function ${\mfk q}' : {\msr A}  \rightarrow \mathbb{N}\cup \{-\infty\}$ (cf. \eqref{2.2}--\eqref{2.3}) by ${\mfk q}' (x^\al y^\be)=    \alpha_{n_2} +  \be_{n_2} $.
Then $ (x_0 y_{i_1}(\td{\Delta}+\ptl_{x_{n_2}}\ptl_{y_{n_2}}) (x_{n_2} y_{n_2} g_1))$ is $\mfk{q}'$-homogeneous of degree ${\mfk q}' (x^\al y^\be) +2$.

For fixed weight $\ell'+k$, every monomial in this case satisfies $\mfk q'\leq\ell'+k-3$, because $x_0y_{i_1}$ already has weight three. Apply descending induction on $\mfk q'$. Each nonzero remainder in \eqref{6.23} has the same weight and $\mfk q'$ increased by two. At the maximal value no such remainder exists. The other two terms in \eqref{6.23} act on $V'_{k-1}$, and the coefficient $(2\alpha_{n_2}+3)(\beta_{n_2}+1)$ is nonzero in characteristic zero. This proves the required inclusion.\pse

(4) Under the conditions $\al_0>0$, $n_1<n_2$, and $\sum_{i=n_2+1}^{n} \al_i>0$, the situation is symmetric to Case (3).\pse

(5) When $ \sum_{i=n_2+1}^n \al_i=\sum_{i=1}^{n_1} \be_i =0$, $\sum_{i= 1}^{n_1}\al_i >0$ and $\sum_{i=n_1+1}^{n_2}\al_i >0$, we write $x^\al y^\be=x_{i_1}x_{i_2} g_1$, where $i_1\in J_1$, $i_2\in J_2$ and $T'(g_1)\in TN'(k-1)$.

By \eqref{6.5}, we have
\begin{align*} \begin{split}
\td{E}_{i_2,i_1}(T'(- g_1))
=&T'\left( (-x_{i_1}x_{i_2}-y_{i_1}\ptl_{y_{i_2}})(- g_1)\right)\\
=&T'(x^\al y^\be)+T'(y_{i_1}\ptl_{y_{i_2}}(g_1)).
\end{split}\end{align*}
If $\ptl_{y_{i_2}}(g_1)=0$, $T'(x^\al y^\be)=\td{E}_{i_2,i_1}(T'(- g_1))$ yields that $T'(x^\al y^\be) \in U_1(\mfk{o}(2n+1))( V_{k-1}')$.
Suppose that $\ptl_{y_{i_2}}(g_1)\neq 0$, then $\mfk d'_1(y_{i_1}\ptl_{y_{i_2}}(g_1))=\mfk d'_1(x^\al y^\be)$. Hence $T'(y_{i_1}\ptl_{y_{i_2}}(g_1))\in TN'(k)$.
By Case (1) and Case (3), we have shown that $T'(y_{i_1}\ptl_{y_{i_2}}(g_1)) \in U_1(\mfk{o}(2n+1))( V_{k-1}')$.
Therefore, $T'(x^\al y^\be) =\td{E}_{i_2,i_1}(T'(- g_1))-T'(y_{i_1}\ptl_{y_{i_2}}(g_1)) \in U_1(\mfk{o}(2n+1))( V_{k-1}')$.\pse

(6) When $ \sum_{i=n_2+1}^n \al_i=\sum_{i=1}^{n_1} \be_i =0$, $\sum_{i=n_2+1}^{n}\be_i >0$ and $\sum_{i=n_1+1}^{n_2}\al_i >0$, we write $x^\al y^\be=y_{i_3}x_{i_2} g_1$, where $i_3\in J_3$, $i_2\in J_2$ and $T'(g_1)\in TN'(k-1)$.

By \eqref{6.5}, we have
\begin{align*} \begin{split}
\td{E}_{i_3,n+i_2}(T'( g_1))
=&T'\left( (y_{i_3}x_{i_2}+ x_{i_3}\ptl_{y_{i_2}})( g_1)\right)\\
=&T'(x^\al y^\be)+T'(x_{i_3}\ptl_{y_{i_2}}(g_1)).
\end{split}\end{align*}
If $\ptl_{y_{i_2}}(g_1)=0$, $T'(x^\al y^\be)=\td{E}_{i_3,n+i_2}(T'(g_1))$ yields that $T'(x^\al y^\be) \in U_1(\mfk{o}(2n+1))( V_{k-1}')$.
Suppose that $\ptl_{y_{i_2}}(g_1)\neq 0$, then $\mfk d'_1(x_{i_3}\ptl_{y_{i_2}}(g_1))=\mfk d'_1(x^\al y^\be)$. Hence $T'(x_{i_3}\ptl_{y_{i_2}}(g_1))\in TN'(k)$.
By Case (2) and Case (4), we have shown that $T'(x_{i_3}\ptl_{y_{i_2}}(g_1)) \in U_1(\mfk{o}(2n+1))( V_{k-1}')$. \pse

(7) When $ \sum_{i=n_2+1}^n \al_i=\sum_{i=1}^{n_1} \be_i =0$, $\sum_{i=  1}^{n_1 }\al_i >0$ and $\sum_{i=n_1+1}^{n_2}\be_i >0$, we write $x^\al y^\be=x_{i_1}y_{i_2} g_1$, where $i_1\in J_1$, $i_2\in J_2$ and $T'(g_1)\in TN'(k-1)$.

By \eqref{6.5}, we have
\begin{align*} \begin{split}
\td{E}_{n+i_1,i_2}(T'(  g_1))&=T'\left( (x_{i_1}y_{i_2} +y_{i_1}\ptl_{x_{i_2}})(g_1) \right)\\
&=T'( x^\al y^\be)+T'(y_{i_1}\ptl_{x_{i_2}} (g_1)).
\end{split}\end{align*}
If $\ptl_{x_{i_2}} (g_1)=0$, $T'(x^\al y^\be)=\td{E}_{n+i_1,i_2}(T'(  g_1))$ yields that $T'(x^\al y^\be) \in U_1(\mfk{o}(2n+1))( V_{k-1}')$.
Suppose that $\ptl_{x_{i_2}} (g_1)\neq 0$, then $\mfk d'_1(y_{i_1} \ptl_{x_{i_2}} (g_1) )=\mfk d'_1(x^\al y^\be)$. Hence $T'(y_{i_1} \ptl_{x_{i_2}} (g_1) )\in TN'(k)$.
By Case (1) and Case (3), we have shown that $T'(y_{i_1}\ptl_{x_{i_2}}(g_1)) \in U_1(\mfk{o}(2n+1))( V_{k-1}')$.
Therefore, $T'(x^\al y^\be) =\td{E}_{n+i_1,i_2}(T'(  g_1))-T'(y_{i_1}\ptl_{x_{i_2}} (g_1))\in U_1(\mfk{o}(2n+1))( V_{k-1}')$.\pse

(8) When $ \sum_{i=n_2+1}^n \al_i=\sum_{i=1}^{n_1} \be_i =0$, $\sum_{i= n_2+ 1}^{n  }\be_i >0$ and $\sum_{i=n_1+1}^{n_2}\be_i >0$, we write $x^\al y^\be=y_{i_2}y_{i_3} g_1$, where $i_2\in J_2$, $i_3\in J_3$ and $T'(g_1)\in TN'(k-1)$.

By \eqref{6.5}, we have
\begin{align*} \begin{split}
\td{E}_{ i_3,i_2}(T'(  g_1))&=T'\left( (y_{i_3}y_{i_2} +x_{i_3}\ptl_{x_{i_2}})(g_1) \right)\\
&=T'( x^\al y^\be)+T'(x_{i_3}\ptl_{x_{i_2}} (g_1)).
\end{split}\end{align*}
If $\ptl_{x_{i_2}} (g_1)=0$, $T'(x^\al y^\be)=\td{E}_{ i_3,i_2}(T'(  g_1))$ yields that $T'(x^\al y^\be) \in U_1(\mfk{o}(2n+1))( V_{k-1}')$.
Suppose that $\ptl_{x_{i_2}} (g_1)\neq 0$, then $\mfk d'_1(x_{i_3} \ptl_{x_{i_2}} (g_1) )=\mfk d'_1(x^\al y^\be)$. Hence $T'(x_{i_3} \ptl_{x_{i_2}} (g_1) )\in TN'(k)$.
By Case (2) and (4), we have shown that $T'(x_{i_3}\ptl_{x_{i_2}}(g_1)) \in U_1(\mfk{o}(2n+1))( V_{k-1}')$.
Therefore, $T'(x^\al y^\be) =\td{E}_{ i_3,i_2}(T'(  g_1))-T'(x_{i_3}\ptl_{x_{i_2}} (g_1))\in U_1(\mfk{o}(2n+1))( V_{k-1}')$.\pse

(9) When $\al_0 > 0$ and $n_1=n_2$, we have $T'(x^\al y^\be)  \in U_1(\mfk{o}(2n+1))( V_{k-1}')$ by Lemma \ref{lem:6.1}.
\psp

(10) The remaining case has $\alpha_0=1$, no variable $y_i$ with $i\in J_1$, no variable $x_j$ with $j\in J_3$, and no middle variable. Write the monomial as $x_0f$, with $f$ in the variables $x_i$ ($i\in J_1$), $y_j$ ($j\in J_3$). Its weight is one, so $k\geq1$ forces $\ell\leq0$ and $k=1$. Since $\deg f=1-\ell>0$, choose a variable dividing $f$. If it is $x_i$, then
\[
T'(x_0f)=x_0f=\widetilde E_{0,i}(-f/x_i),\qquad f/x_i\in V'_0.
\]
The case of a divisor $y_j$ uses $\widetilde E_{0,n+j}$ in the same way.

These cases exhaust the initial-data monomials: if no outer positive-weight variable occurs, the signed degree and $k\geq1$ force either a middle variable together with an outer negative-weight variable, or the case just considered. Thus the cases (5)--(8) and (10) cover that possibility.

In summary, we have $ TN'(k) \subseteq  U_1(\mfk{o}(2n+1))( V_{k-1}')$ for $k\geq 1$. Then $V_k'=M_k$ for integer $k\geq 1$ by Lemma \ref{lem:2.1}.
\end{proof}

\section*{Statements and Declarations}

{\bf Data Availability Statement}
 
This is a theoretical study. No datasets were generated or analysed. All results are established by proofs contained within the article. 

\psp

{\bf Conflict of interest}

The author declares that there are no conflicts of interest regarding the publication of this paper.

\psp

{\bf Funding}

This research did not receive any specific grant from funding agencies in the public, commercial, or not-for-profit sectors.

\psp

{\bf Use of AI-assisted tools}

During the revision, ChatGPT (OpenAI) was used to assist with the review and revision of mathematical arguments, as well as with language and LaTeX editing. The author takes responsibility for the content of the manuscript.

\psp

 \end{document}